\documentclass{article}
\usepackage[numbers]{natbib}
\usepackage{circuitikz}
\usepackage{float}
\usepackage{subcaption}
\usepackage{amsmath, hyperref, a4wide}
\usepackage{amsthm, amsfonts}
\def\tsc#1{\csdef{#1}{\textsc{\lowercase{#1}}\xspace}}
\tsc{WGM}
\tsc{QE}
\usepackage{authblk}
\newtheorem{theo}{Theorem}
\newtheorem{lemma}[theo]{Lemma}

\newtheorem{prop}[theo]{Proposition}

\newtheorem{claim}{Claim}

\newcommand{\CSP}[1]{\mbox{\rm CSP$(#1)$}}

\begin{document}

\title{The majority of a cactus is prickly.}

\author[1]{Catarina Carvalho}
\author[1]{Mathison Knight}
\author[2]{Barnaby Martin}

\affil[1]{Department of Physics, Astronomy and Mathematics, University of Hertfordshire}
\affil[2]{Department of Computer Science, Durham University}

\maketitle

\begin{abstract}
We characterise all (undirected) cacti graphs, and forests of cacti, with an edge-preserving  near-unanimity operation, by means of forbidden induced subgraphs. Indeed all such graphs  also have an edge-preserving  majority operation. We show that said  majority can be chosen to be symmetric and explicitly construct it for all forests of cacti without the forbidden induced subgraphs.
\end{abstract}

\section{Introduction}

Polymorphisms are fundamental combinatorial objects whose algebraic properties play a key role in the study of the computational complexity of Constraint Satisfaction Problems (CSPs) and their variants \cite{BulatovJK05, barto2017polymorphisms}. On graphs $H$, polymorphisms are just edge-preserving $k$-ary operations, which may also be seen as homomorphisms from the (direct, or tensor) power $H^k$ to $H$, for some $k\ge 2$.

 At the point the machinery of polymorphisms started to be used on arbitrary relational structures, it was noted that some special edge-preserving operations were known to be tied to the tractability of some graph homomorphism problems. In many cases, the original work did not explain the tractability in terms of polymorphisms, but this explanation subsequently became perfectly natural.
 For example, if a graph $H$ has something called the $\underbar{X}$-property then there is a linear time algorithm to solve the CSP on $H$, $\CSP H$ (or equivalently, the homomorphism problem of $H$, HOM($H$) \cite{Xunderbar}. The $\underbar{X}$-property can be seen as  the existence of a ${\emph minimum}$ polymorphism. 
  Historically, a fundamental role has been played by 
 \emph{majority} operations, which are ternary operations $m:H^3\longrightarrow H$ that satisfy the following identities
$$m(x, y, y)=m(y, x, y)=m(y, y, x)=y$$
for all vertices $x, y$ of a graph $H$.  
Majority operations were studied in a purely combinatorial context as far back as \cite{Bandelt1992GraphsWE}, pre-dating completely the modern study of CSPs, and have played a key role in the characterisation of the complexity of several CSP related problems, e.g., \cite{schaefersatisfiability, federvardi, dalmau2008majority, dehaan_et_al:LIPIcs.APPROX/RANDOM.2025.19, bulatovconservative}. \emph{Near-unanimity (NU)} operations $n:H^k\rightarrow H$ generalise majority operations to higher arities, and are $k$-ary operations  satisfying the identities
$$n(x, y, y,\ldots, y)=n(y, x, y, \ldots,y)=\cdots =n(y, y,y, \ldots, x)=y.$$
Call an NU operation of arity $k$ a $k$-NU. NUs also play an important role in the study of the computational complexity of CSP related problems, e.g. \cite{ bartoNU, dalmaurobust, DawarNU}. 

While the notion of a majority polymorphism, and other types of operations, is natural from an algebraic point of view,  for many  the question remains  on how they can be interpreted from a  graph theoretical point of view, e.g., for a given type of polymorphism what (classes of) graphs are preserved by it?
Work has been done in the topic since the  90's, e.g., \cite{Bandelt1992GraphsWE, carvalhomaltsev,  kazda2011maltsev, Larose03:nugraphs, hell2004graphs, quackenbush1990clones, Larose17, DBLP:conf/latin/HellRR24} but for many polymorphisms the question remains  open. 

In this paper we answer this question for the majority, and more generally NU, polymorphisms in the class of (forests of) cacti graphs.

\begin{theo}\label{theomajority}
A forest of cacti is preserved by an NU operation  if and only if it is loop-connected and it does not contain any of the following as an induced subgraph:
\begin{enumerate}
\item a cycle of length greater than or equal to 5;
\item an irreflexive triangle;
\item a square with at least  3 reflexive vertices.
\end{enumerate}
Furthermore if it does not have these induced subgraphs the forest of cacti is preserved by a symmetric majority polymorphism (i.e. a $3$-NU operation).
\end{theo}
 A majority  operation $m$ is called {\it symmetric} if it satisfies the identity
$$m(a_1, a_2,a_3) = m (a_{\pi(1)}, a_{\pi(2)},  a_{\pi(n)})$$
for all permutations $\pi$  of $\{1, 2, 3 \}$.

This is, as far as we are aware, the first characterisation for graphs with some loops allowed with an edge-preserving majority operation. If we restrict ourselves to  reflexive graphs (i.e. with loops in all vertices), it was proved in \cite{jawhari1986retracts} that
a reflexive graph is an absolute retract if and only if it admits an edge-preserving
majority operation. Also it was shown in \cite{Bandelt1992GraphsWE} that   an irreflexive graph  (i.e. with no loops allowed) admits an edge-preserving majority function  if and only if it is an absolute retract of a bipartite graph. 
Characterisations have also been found for reflexive and irreflexive graphs preserved by   NU operations  \cite{feder2013graphs, feder2014graphs, nureflexive, jaffenu}, inspired by a previous characterisation of partial ordered sets  which have order-preserving majority functions \cite{jawhari1986retracts}. Also, it is known that if a graph  $H$ admits a majority function then $H$ has treewidth two duality \cite{hell2004graphs}.

Even stating the characterisation as to which irreflexive graphs have NU is somewhat non-trivial. This prefigures many challenges in the general case in which loops are allowed but not forced. Our approach to this case is bottom up, considering simple forms and seeing how  we might generalise. The starting point may be seen as paths or trees. It was announced in the conference publication \cite{martin2011qcsp} that loop-connected trees have a majority polymorphism, but the given majority was specified ambiguously and the proof was only sketched. In this work, we clarify that result and extend it, through pseudo-trees, to {\it forests of cacti}, i.e.  graphs in which any two simple cycles have at most one vertex in common, and to NU operations of any arity.

Our characterisation does not stretch so far beyond cacti. Let $H$ be $K_4$ with a single looped vertex. Notwithstanding that this has an irreflexive $K_3$ as an induced subgraph, it does have a majority polymorphism.\footnote{In some sense $H$ is minimally non-cactus, since any two cycles share at most two edges. The class of graphs whose cycles share at most a single edge actually coincides with cacti, since two cycles joined on a single edge create a large cycle with at least two edges in common with either of the smaller cycles.}

We show that all cacti with an NU operation actually have a $3$-NU operation (a majority). This is not true for all graphs, for example the Haj\'os graph, or 2-sun graph,  with loops added to all vertices is preserved by a 4-NU  but not by a majority operation. Indeed, for each $k\geq 3$ there is a graph that is preserved by a $(k+1)$-NU but not by a $k$-NU \cite{brewster2008near}. 

We give a construction of our majority operation, and note that  this is not typical in the literature. For example, in \cite{feder2003bi}, a construction of the majority is given only for bi-arc trees, and the authors mention that the description of a majority function for all bi-arc graphs  does not appear to be an easy task. Furthermore the majority that we define is symmetric. It is not the case that all (di)graphs with a majority polymorphism have a symmetric majority polymorphism, for example the directed, irreflexive,  triangle has  an edge preserving  majority  operation but not a symmetric majority.

\section{Preliminaries and related work}

In this article, unless mentioned otherwise,  graphs are undirected, simple, and possibly with some loops, i.e. we allow some vertices to be adjacent to themselves. 

A homomorphism $f: G\longrightarrow H$ of a graph $G$ to a graph $H$ is a vertex mapping
which preserves edges, i.e. a mapping $f: V(G)\longrightarrow V(H)$ such that $(u, v)\in  E(G)$
implies $(f(u),  f(v))\in  E(H)$. 

Given a graph $H$, the $H$-colouring problem, or {\it graph homomorphism problem HOM$(H)$} takes as input a graph $G$ and asks if $G$ admits an homomorphism to $H$. This problem is equivalent to the Constraint Satisfaction Problem with the fixed target $H$, CSP($H$). Given a graph $G$ and lists $L(v)\subseteq V(H)$ for
all $v \in V(G)$, a list homomorphism of $G$ to $H$, with respect to the lists $L$, is
a homomorphism $f: G\longrightarrow H$ with $f(v)\in L(v)$ for each $v\in  V(G)$. 
The {\it list homomorphism problem L-HOM($H$)} asks, for input $G$ with lists $L$,
whether there is a list homomorphism of $G$ to $H$ with respect to $L$. The {\it retraction problem  RET($H$)} takes as input a graph $G$ and asks if $G$ is retractable to a fixed graph $H$. 

The algebraic approach to the study of the computational complexity of CSPs is driven through a Galois connection between operations closed under superposition and relations closed under definability in the logic that contains just $\exists$, $\wedge$ and $=$ (called \emph{primitive positive} (pp) logic). Specifically, the relations that are preserved by the polymorphisms of $H$ are precisely those that are pp-definable in $H$ \cite{geiger1968closed,bodnarchuk1969galois1}. It follows that if one can pp-define some relation in $H$ that is violated by some operation, then that operation cannot be a polymorphism of $H$.

The survey \cite{Hell06:there} provides a good first  view of the cross fertilisation of ideas from looking at CSPs as graph homomorphism problems and vice-versa, including an initial idea of the  importance of polymorphisms in the study of graph problems. More recent, the survey \cite{Larose17} provides some of  the key results on the interplay between algebraic and graph-theoretic methods in the study of graph homomorphism problems.
Naturally,  in  this interplay lies the question on how an algebraic idea, namely a polymorphism,  can be interpreted in graph-theoretical terms. 
We briefly  mention the types of polymorphisms for which some graph-theoretical characterisation is already know. 

A majority choice function, nowadays better known  as a {\it conservative  majority}, is a majority operation $m$ with the added requirement that $m(x, y, z)\in \{x, y, z\}$, for all vertices $x, y, z\in H$. Feder and Vardi \cite{federvardi} proved that if a graph $H$ has a conservative majority polymorphism then the  L-HOM($H$) is tractable, i.e. belongs to the complexity class $\mathbf{P}$. 
We now know that in fact it is in the complexity class {\bf{NL}} \cite{dalmau2008majority}, this is also the case for HOM($H$) and RET($H$), when $H$ has a  majority polymorphism. For these problems we do not need the majority to  be conservative.  For more details on computational and descriptive complexity of graph homomorphism problem and the interplay with the CSP we refer the reader  to \cite{Hell06:there, Larose17}.

Conservative majorities  were further studied in  \cite{feder2003bi}, where the authors showed that the graphs that have a conservative majority polymorphism are the  bi-arc graphs, which are  exactly  the ones with tractable List Homomorphism problem. This conservative  majority polymorphism is not explicitly presented in the paper, but it was described for trees in \cite{feder2007structure}. When we restrict ourselves to reflexive graphs the bi-arc graphs  correspond to interval graphs \cite{hell2004graphs} and for them a majority operation can be taken as the median in the order defined by the interval representation of the graph \cite[Theorem 4.1]{feder2003bi}. 
This leaves open the question on how to describe the conservative  majority polymorphism for bi-arc graphs that are neither trees nor fully reflexive,  but we note that in \cite{feder2003bi} the authors mentioned that 
``One may try to define a majority choice function for any bi-arc graph H; this
however does not appear to be an easy task.''
For directed graphs (digraphs) it was shown that they  admit a conservative majority function if and only if they have no permutable triple \cite{doi:10.1137/1.9781611973082.131}. While this is an interesting and useful result in proving the complexity of List Homomorphism problems, it still leaves open the question of which (classes of) digraphs  have permutable triples, and what the majority function will look like for them.

A ternary operation $f:H^3\rightarrow H$ is \emph{Maltsev} if it satisfies $f(x, y, y)=f(y, y, x)=x$
for all vertices $x, y$ of a graph $H$. The digraphs which have a Maltsev polymorphism have been classified in \cite{carvalhomaltsev}. It turns out that  all digraphs with a Maltsev polymorphism also have a majority polymorphism \cite{kazda2011maltsev}.

  A binary operation is a  \textit{semilattice} if it is  idempotent, symmetric and associative. Reflexive graphs which have a certain type of semilattice polymorphism, namely a non-crossing semilattice polymorphism, were shown to be chordal \cite{hell2014semilattice}.  For reflexive graphs to be preserved by a conservative semilattice  they have to be interval graphs \cite{FederH98}.  Also, it is known that if a graph $G$ has a non-crossing semilattice polymorphism of width $k-1$, then $G$ also has a $k+1$-NU polymorphism \cite{siggers2018reflexive}.

\paragraph*{Paper and Proof Outline}

A {\it cactus} is a connected graph in which any two simple cycles have at most one vertex in common. To a graph where each connected component is a cactus we assign the name {\it forest of cacti}.

A graph is {\it loop-connected} if its reflexive vertices induce a connected subgraph. In \cite{FederHJKN10} the authors proved that if a graph $H$ is not loop-connected, i.e. is {\it  loop-disconnected}, or has cycles of length greater than or equal to five,   then its retraction problem is {\bf NP}-complete. Assuming that ${\mathbf{P}\neq \mathbf{NP}}$, this means that a graph that is not loop-connected cannot be preserved by a majority polymorphism. For cacti we present a proof that does not rely on complexity theoretical assumptions.  

We show that  forests of cacti that have  a $k$-NU polymorphism, for any $k$,  are in fact preserved by a symmetric majority, and  explicitly construct said  majority. 
We note that for trees, a similar majority to the one we present was proposed in  \cite{martin2011qcsp}, without a proof. There, the existence of majority was used to characterise the tree templates whose Quantified Constraint Satisfaction Problem is in {\bf{NL}}, but the definition given was somewhat ambiguous and does not generalise to cacti containing cycles of length four.

In Section~\ref{omitting} we show that a  cacti containing any of the forbidden induced subgraphs mentioned in Theorem~\ref{theomajority} is not preserved by a $k$-NU for any $k$.  We do this  without the use of complexity theoretical assumptions.

In Section~\ref{maj} we  construct  a symmetric majority operation that preserves the remainder of the cacti, i.e. the ones without the induced forbidden subgraphs. Then we prove that the majority is indeed edge-preserving, and  extend it to forests of cacti. 
This majority function constructed is,  one might say, a bit prickly. To define it, we consider different components of a graph, the reflexive component of which one might think of as the ``body'' of the graph and the irreflexive parts as ``limbs'' that stick out from the body of the graph in all possible directions (like an Hecatoncheires).
The majority is then constructed depending on the position of the vertices in the graph. For the vertices only in the body the majority behaves just like the median, which was the known majority function for reflexive interval graphs; when the three  vertices to which we are applying the function  are all in  different limbs (or in the different limbs and the body) the  result of the majority function, just like the median,  will be in the body; and when two or three of the vertices that we are applying the function to  are in the same limb we construct a majority by expanding a known majority based on the vertices' locations with respect to each other, giving four categories (Case 1-Case 4) of analysis. From a bird's eye view the majority function projects far apart vertices in the limbs to the body of the graph.  The proof that the majority we define is indeed edge-preserving is definitely prickly, as we need to address each case individually and go through the analysis of many subcases. To make it easier for the reader to follow the argument we break the proof into may subsections, where each subsection zooms in on a different placement of the vertices within the limbs of the graph and looks at all possible behaviours of the adjacent vertices.

\section{Cacti omitting majority}\label{omitting}

In this section we start by showing that if a  cactus is loop-disconnected then it is not preserved by a majority operation. We then show that cacti  containing any of the induced subgraphs mentioned in Theorem~\ref{theomajority} are  not preserved by a $k$-NU operation for  any $k$. Our proofs do not rely on complexity theoretical assumptions.

\newcommand{\cyl}{\ensuremath{\mathrm{Cyl}}}

\subsection{Loop-disconnected cactus}

We note that, under complexity theoretic assumptions, it  is already known that loop-disconnected pseudo-forests cannot be preserved by a majority polymorphism, as  in \cite{FederHJKN10} the authors show that the retraction problem of a pseudo-forest $H$,  $RET(H)$, is {\bf{NP}}-complete if there is a connected component $H'$ in $H$
such that the looped vertices in $H'$ induce a disconnected graph.  

We avoid the theoretical complexity assumption, by adapting \cite[Proposition 2]{martin2011qcsp} to show that loop-disconnected cacti are not preserved by a $k$-NU operation for any $k$. As the arity does not play a role in the next result we omit it. This will also be the case in what follows, so when saying NU we mean $k$-NU for any $k\ge 3$.

\begin{prop}
    
Let $H$ be a cactus graph that is loop-disconnected, then $H$ does not admit an NU polymorphism.
\end{prop}
\begin{proof}
Let $H$ be a connected cactus graph and let $p, p+1, \ldots, q$ be the shortest loop-disconnected path in $H$, call this $P$. That is, $p, q$ are reflexive vertices, and $p+1, p+2, \ldots, q-1$ are irreflexive vertices, with $(p, p+1), (p+1, p+2), \ldots, (q-1, q)$ undirected edges of $H$. Note that $p$ may have some neighbours, say  $s$,  distinct from $p+1$, but there can never be a path from $s$ to $q$ that is shorter than $q-p-1$.

Assume, for a contradiction, that $H$ is preserved by an NU operation $n$ of any arity. We have that $n(p,\ldots,p, q)=p$
and $n(p,\ldots,p, p+1, q)$ and $n(p, \ldots,p, p+1, p, q)$ are both adjacent to $p$ and to each other. Let $n(p,\ldots,p, p+1, q)=r$, and $n(p, \ldots,p+1, p, q)=r'$. Note that these cannot both be equal to $p+1$, as $p+1$ is irreflexive.

There is a path of length at most $q-p$ from $r$ to $q=n(q,\ldots ,q,p,q)$ (or $r'$ to $q=n(q,\ldots ,p,q,q)$). As $d(r,q)\leq q-p$ and $d(r',q)\leq q-p$, if $r\neq r'$ then, assuming that $r$ is on the path  $P$, there is a different path from $r'$ to $q$ that is not via $r$. Hence, the edge $(r,r')$ is in two cycles, a contradiction.

Thus, $r=r'$, but $r$ is an irreflexive vertex, a contradiction.

\end{proof}

\subsection{Big cycles}

The arguments in this section come from constructions underlying the work of Feder, Hell and Huang \cite{FederH98,FederHH99}. Let us note that the absence of an NU polymorphism for the graphs of this section already follows modulo complexity-theoretic assumptions from \cite{FederHJKN10}. Our task is to make the proof independent of those assumptions.

Let $H$ be a cycle of length $k \geq 3$, which may contain looped vertices. Suppose $V(H)=\{0,\ldots,k-1\}$. Define the ``cylinder'' graph $\cyl_m(H)$ to be built from $m$ copies of the underlying irreflexive graph of $H$, indexed from $0$ to $m-1$, with an edge between vertex $j$ of the $i^{th}$ copy and vertex $j$ of the $(i+1)^{th}$ copy, for all $i \in \{0,\ldots,m-1\},j \in \{0,\ldots,m-2\}$. 

In graph theoretical terms, $\cyl_m(H)$ is the Cartesian product of the path on $m$ vertices $P_m$  with the underlying irreflexive $H$ (and the coordinates will be in that order). Call the copy of $H$ with index $0$ in $\cyl_m(H)$ the \emph{base copy}, $H_0$, and the copy with index $m-1$ the \emph{top copy}.

\subsubsection{Odd cycle with at least one vertex loopless}
\label{sec:odd-cycle-with-at-least-one-vertex-loopless}

\begin{figure}
\begin{center}
    %
%
\setlength{\unitlength}{3947sp}%
\begingroup\makeatletter\ifx\SetFigFont\undefined%
\gdef\SetFigFont#1#2#3#4#5{%
  \reset@font\fontsize{#1}{#2pt}%
  \fontfamily{#3}\fontseries{#4}\fontshape{#5}%
  \selectfont}%
\fi\endgroup%
\begin{picture}(4171,1199)(810,-709)
{\color[rgb]{0,0,0}\thinlines
\put(4662,-269){\circle*{76}}
}%
{\color[rgb]{0,0,0}\put(4183,-241){\circle*{76}}
}%
{\color[rgb]{0,0,0}\put(4043,220){\circle*{76}}
}%
{\color[rgb]{0,0,0}\put(4503,445){\circle*{76}}
}%
{\color[rgb]{0,0,0}\put(4935,163){\circle*{76}}
}%
{\color[rgb]{0,0,0}\put(1475,-269){\circle*{76}}
}%
{\color[rgb]{0,0,0}\put(996,-241){\circle*{76}}
}%
{\color[rgb]{0,0,0}\put(856,220){\circle*{76}}
}%
{\color[rgb]{0,0,0}\put(1316,445){\circle*{76}}
}%
{\color[rgb]{0,0,0}\put(1748,163){\circle*{76}}
}%
{\color[rgb]{0,0,0}\put(3055,-277){\circle*{76}}
}%
{\color[rgb]{0,0,0}\put(2576,-249){\circle*{76}}
}%
{\color[rgb]{0,0,0}\put(2436,212){\circle*{76}}
}%
{\color[rgb]{0,0,0}\put(2896,437){\circle*{76}}
}%
{\color[rgb]{0,0,0}\put(3328,155){\circle*{76}}
}%
{\color[rgb]{0,0,0}\put(4071,247){\line( 5, 2){428.621}}
}%
{\color[rgb]{0,0,0}\multiput(4041,223)(2.65439,-7.96316){58}{\makebox(1.6667,11.6667){\tiny.}}
\put(4189,-232){\line( 6,-1){468.486}}
\put(4663,-277){\line( 3, 5){269.559}}
\put(4933,172){\line(-3, 2){432.692}}
}%
{\color[rgb]{0,0,0}\put(884,247){\line( 5, 2){428.621}}
}%
{\color[rgb]{0,0,0}\put(1306,463){\line( 1, 0){348}}
}%
{\color[rgb]{0,0,0}\put(1720,191){\line( 1, 0){338}}
\put(2058,191){\line(-1, 0){141}}
}%
{\color[rgb]{0,0,0}\put(996,-232){\line( 6, 1){329.351}}
}%
{\color[rgb]{0,0,0}\put(1466,-260){\line( 5, 1){290.769}}
}%
{\color[rgb]{0,0,0}\put(846,229){\line( 6, 1){326.108}}
}%
{\color[rgb]{0,0,0}\multiput(854,223)(2.65439,-7.96316){58}{\makebox(1.6667,11.6667){\tiny.}}
\put(1002,-232){\line( 6,-1){468.486}}
\put(1476,-277){\line( 3, 5){269.559}}
\put(1746,172){\line(-3, 2){432.692}}
}%
{\color[rgb]{0,0,0}\put(2464,239){\line( 5, 2){428.621}}
}%
{\color[rgb]{0,0,0}\put(2886,455){\line( 1, 0){348}}
}%
{\color[rgb]{0,0,0}\put(3300,183){\line( 1, 0){338}}
\put(3638,183){\line(-1, 0){141}}
}%
{\color[rgb]{0,0,0}\put(2576,-240){\line( 6, 1){329.351}}
}%
{\color[rgb]{0,0,0}\put(3046,-268){\line( 5, 1){290.769}}
}%
{\color[rgb]{0,0,0}\put(2426,221){\line( 6, 1){326.108}}
}%
{\color[rgb]{0,0,0}\multiput(2434,215)(2.65439,-7.96316){58}{\makebox(1.6667,11.6667){\tiny.}}
\put(2582,-240){\line( 6,-1){468.486}}
\put(3056,-285){\line( 3, 5){269.559}}
\put(3326,164){\line(-3, 2){432.692}}
}%
\put(2618,-649){\makebox(0,0)[lb]{\smash{{\SetFigFont{8}{9.6}{\rmdefault}{\mddefault}{\updefault}{\color[rgb]{0,0,0}Copy k+1}%
}}}}
\put(4260,-617){\makebox(0,0)[lb]{\smash{{\SetFigFont{8}{9.6}{\rmdefault}{\mddefault}{\updefault}{\color[rgb]{0,0,0}Copy 0}%
}}}}
\put(2445,-219){\makebox(0,0)[lb]{\smash{{\SetFigFont{8}{9.6}{\rmdefault}{\mddefault}{\updefault}{\color[rgb]{0,0,0}s}%
}}}}
\put(848,-226){\makebox(0,0)[lb]{\smash{{\SetFigFont{8}{9.6}{\rmdefault}{\mddefault}{\updefault}{\color[rgb]{0,0,0}t}%
}}}}
\put(1085,-656){\makebox(0,0)[lb]{\smash{{\SetFigFont{8}{9.6}{\rmdefault}{\mddefault}{\updefault}{\color[rgb]{0,0,0}Copy 2k}%
}}}}
\put(2169,281){\makebox(0,0)[lb]{\smash{{\SetFigFont{8}{9.6}{\rmdefault}{\mddefault}{\updefault}{\color[rgb]{0,0,0}...}%
}}}}
\put(2162,-14){\makebox(0,0)[lb]{\smash{{\SetFigFont{8}{9.6}{\rmdefault}{\mddefault}{\updefault}{\color[rgb]{0,0,0}...}%
}}}}
\put(2162,-14){\makebox(0,0)[lb]{\smash{{\SetFigFont{8}{9.6}{\rmdefault}{\mddefault}{\updefault}{\color[rgb]{0,0,0}...}%
}}}}
\put(3746,-59){\makebox(0,0)[lb]{\smash{{\SetFigFont{8}{9.6}{\rmdefault}{\mddefault}{\updefault}{\color[rgb]{0,0,0}...}%
}}}}
\put(3759,281){\makebox(0,0)[lb]{\smash{{\SetFigFont{8}{9.6}{\rmdefault}{\mddefault}{\updefault}{\color[rgb]{0,0,0}...}%
}}}}
\end{picture}%
\end{center}
    \caption{The two cylinder construction of $\cyl^+(H)$ from Section~\ref{sec:odd-cycle-with-at-least-one-vertex-loopless}. All vertices are loopless.}
    \label{fig:C5s}
\end{figure}
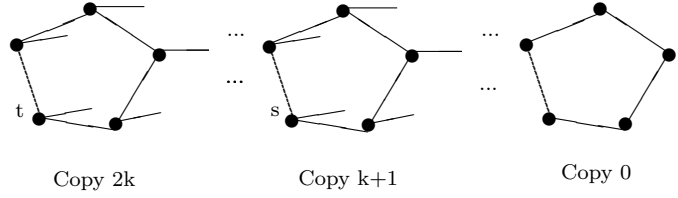

\noindent If $k\geq 5$ is odd, and $H$ contains at least one loopless vertex, let $\cyl^+(H)$ be two cylinders $\cyl_{k+1}(H)$ and $\cyl_{k-1}(H)$ joined (vertically) in series. That is, the base copy of $H$ in the upper cylinder is identified with the top copy of $H$ in the lower cylinder. Label vertex $(0,0)$ in the upper cylinder (i.e. $(k,0)$ in the lower cylinder) as $s$ and label vertex $(k-2,0)$ in the upper cylinder as $t$. Let $H_0$ refer to the base copy of $H$ in the lower cylinder. The shortest path between $s$ and $t$ is of length $k-2$. This gadget is drawn in Figure~\ref{fig:C5s}, and is rotated 90 degrees counter clockwise for sake of space.
\begin{lemma}
    Let $H$ be a  cycle of odd length $k \geq 5$ with at least one loopless vertex.  Let $f$ range over all homomorphisms from $\cyl^+(H)$ to $H$ such that $f(H_0)$ is the identity. Then the orbit of $(s,t)$ in $V(H)^2$ is precisely $s \neq t \in \{0,\ldots,k-1\}$.
\label{lem:no-majority-odd-not-reflexive}
\end{lemma}
\begin{proof}
Without loss of generality, let $0$ be loopless. Consider the copies of $H$ in $\cyl^+(H)$ from the base copy to the top copy. Let us commence by showing that, if the base copy maps according to the identity, then each successive copy of $H$ maps under $f$ to $H$ according to some automorphism of the irreflexive graph underlying $H$ (namely $C_k$). Our argument will be inductive. The base copy is the base case and holds by assumption. Suppose it holds for the $i^{th}$ copy, and imagine that in the $i^{th}$ copy the vertex $(j,i)$ maps to $0$. In the $(i+1)^{th}$ copy, vertex $(j,i+1)$ must map either to $1$ or $k-1$. Moreover, once a map is chosen for this vertex, it will set the position for all of the remaining vertices from this copy. If $f(j,i+1)=1$, the cycle in layer $i+1$ will advance one turn clockwise relative to the cycle in layer $i$; otherwise it will advance one turn anticlockwise.

The fact that $s$, and $t$,  may individually attain any vertex of $H$ now follows from the fact that there is a walk of length $k$ between any distinct vertices of $V(H)$. The fact that $s$ and $t$ cannot map to the same vertex follows from the fact that there is no walk of length $k-2$ between a vertex and itself in the irreflexive graph underlying $H$, namely $C_k$. 
\end{proof}

\subsubsection{Even cycle with at least one vertex loopless}
\label{sec:even-cycle-with-at-least-one-vertex-loopless}

\begin{figure}
\begin{center}
    %
%
\setlength{\unitlength}{3947sp}%
\begingroup\makeatletter\ifx\SetFigFont\undefined%
\gdef\SetFigFont#1#2#3#4#5{%
  \reset@font\fontsize{#1}{#2pt}%
  \fontfamily{#3}\fontseries{#4}\fontshape{#5}%
  \selectfont}%
\fi\endgroup%
\begin{picture}(4245,2260)(366,-1530)
{\color[rgb]{0,0,0}\thinlines
\put(873,599){\circle*{76}}
}%
{\color[rgb]{0,0,0}\put(412,346){\circle*{76}}
}%
{\color[rgb]{0,0,0}\put(891,-255){\circle*{76}}
}%
{\color[rgb]{0,0,0}\put(431,-11){\circle*{76}}
}%
{\color[rgb]{0,0,0}\put(1361, 17){\circle*{76}}
}%
{\color[rgb]{0,0,0}\put(1351,374){\circle*{76}}
}%
{\color[rgb]{0,0,0}\put(404,344){\line( 0,-1){350}}
\put(416, -6){\line( 2,-1){475.600}}
\put(888,-251){\line( 5, 3){462.500}}
\put(1349, 29){\line( 0, 1){350}}
\put(1349,379){\line( 0, 1){  0}}
\put(1349,379){\line(-2, 1){471.200}}
\put(871,601){\line(-5,-3){459.412}}
}%
{\color[rgb]{0,0,0}\put(2477,-631){\circle*{76}}
}%
{\color[rgb]{0,0,0}\put(2016,-884){\circle*{76}}
}%
{\color[rgb]{0,0,0}\put(2495,-1485){\circle*{76}}
}%
{\color[rgb]{0,0,0}\put(2035,-1241){\circle*{76}}
}%
{\color[rgb]{0,0,0}\put(2965,-1213){\circle*{76}}
}%
{\color[rgb]{0,0,0}\put(2955,-856){\circle*{76}}
}%
{\color[rgb]{0,0,0}\put(2008,-886){\line( 0,-1){350}}
\put(2020,-1236){\line( 2,-1){475.600}}
\put(2492,-1481){\line( 5, 3){462.500}}
\put(2953,-1201){\line( 0, 1){350}}
\put(2953,-851){\line( 0, 1){  0}}
\put(2953,-851){\line(-2, 1){471.200}}
\put(2475,-629){\line(-5,-3){459.412}}
}%
\put(1383,-589){\makebox(0,0)[lb]{\smash{{\SetFigFont{8}{9.6}{\rmdefault}{\mddefault}{\updefault}{\color[rgb]{0,0,0}.}%
}}}}
\put(1447,-630){\makebox(0,0)[lb]{\smash{{\SetFigFont{8}{9.6}{\rmdefault}{\mddefault}{\updefault}{\color[rgb]{0,0,0}.}%
}}}}
\put(1529,-677){\makebox(0,0)[lb]{\smash{{\SetFigFont{8}{9.6}{\rmdefault}{\mddefault}{\updefault}{\color[rgb]{0,0,0}.}%
}}}}
\put(1604,-438){\makebox(0,0)[lb]{\smash{{\SetFigFont{8}{9.6}{\rmdefault}{\mddefault}{\updefault}{\color[rgb]{0,0,0}.}%
}}}}
\put(1668,-479){\makebox(0,0)[lb]{\smash{{\SetFigFont{8}{9.6}{\rmdefault}{\mddefault}{\updefault}{\color[rgb]{0,0,0}.}%
}}}}
\put(1750,-526){\makebox(0,0)[lb]{\smash{{\SetFigFont{8}{9.6}{\rmdefault}{\mddefault}{\updefault}{\color[rgb]{0,0,0}.}%
}}}}
\put(3168,-414){\makebox(0,0)[lb]{\smash{{\SetFigFont{8}{9.6}{\rmdefault}{\mddefault}{\updefault}{\color[rgb]{0,0,0}.}%
}}}}
\put(3104,-455){\makebox(0,0)[lb]{\smash{{\SetFigFont{8}{9.6}{\rmdefault}{\mddefault}{\updefault}{\color[rgb]{0,0,0}.}%
}}}}
\put(3022,-502){\makebox(0,0)[lb]{\smash{{\SetFigFont{8}{9.6}{\rmdefault}{\mddefault}{\updefault}{\color[rgb]{0,0,0}.}%
}}}}
\put(3406,-573){\makebox(0,0)[lb]{\smash{{\SetFigFont{8}{9.6}{\rmdefault}{\mddefault}{\updefault}{\color[rgb]{0,0,0}.}%
}}}}
\put(3342,-614){\makebox(0,0)[lb]{\smash{{\SetFigFont{8}{9.6}{\rmdefault}{\mddefault}{\updefault}{\color[rgb]{0,0,0}.}%
}}}}
\put(3260,-661){\makebox(0,0)[lb]{\smash{{\SetFigFont{8}{9.6}{\rmdefault}{\mddefault}{\updefault}{\color[rgb]{0,0,0}.}%
}}}}
{\color[rgb]{0,0,0}\put(4021,604){\circle*{76}}
}%
{\color[rgb]{0,0,0}\put(3560,351){\circle*{76}}
}%
{\color[rgb]{0,0,0}\put(4039,-250){\circle*{76}}
}%
{\color[rgb]{0,0,0}\put(3579, -6){\circle*{76}}
}%
{\color[rgb]{0,0,0}\put(4509, 22){\circle*{76}}
}%
{\color[rgb]{0,0,0}\put(4499,379){\circle*{76}}
}%
{\color[rgb]{0,0,0}\put(3552,349){\line( 0,-1){350}}
\put(3564, -1){\line( 2,-1){475.600}}
\put(4036,-246){\line( 5, 3){462.500}}
\put(4497, 34){\line( 0, 1){350}}
\put(4497,384){\line( 0, 1){  0}}
\put(4497,384){\line(-2, 1){471.200}}
\put(4019,606){\line(-5,-3){459.412}}
}%
{\color[rgb]{0,0,0}\put(427, -7){\line( 1,-1){253}}
}%
{\color[rgb]{0,0,0}\put(876,-269){\line( 1,-1){230}}
}%
{\color[rgb]{0,0,0}\put(1342, 11){\line( 6,-5){247.869}}
}%
{\color[rgb]{0,0,0}\put(1371,361){\line( 1,-1){225}}
}%
{\color[rgb]{0,0,0}\put(864,588){\line( 6,-5){269.803}}
}%
{\color[rgb]{0,0,0}\put(398,314){\line( 1,-1){256}}
}%
{\color[rgb]{0,0,0}\multiput(4514, 69)(-4.69418,-7.04127){42}{\makebox(1.6667,11.6667){\tiny.}}
}%
{\color[rgb]{0,0,0}\multiput(4479,361)(-5.06927,-6.75902){42}{\makebox(1.6667,11.6667){\tiny.}}
}%
{\color[rgb]{0,0,0}\put(4001,629){\line(-3,-5){220.147}}
}%
{\color[rgb]{0,0,0}\multiput(3570, -1)(-5.43560,-6.52272){36}{\makebox(1.6667,11.6667){\tiny.}}
}%
{\color[rgb]{0,0,0}\multiput(4036,-234)(-4.63801,-6.95701){35}{\makebox(1.6667,11.6667){\tiny.}}
}%
{\color[rgb]{0,0,0}\multiput(3558,326)(-5.25511,-6.56889){38}{\makebox(1.6667,11.6667){\tiny.}}
}%
{\color[rgb]{0,0,0}\put(1377,378){\line( 1, 0){2163}}
}%
\put(1424,437){\makebox(0,0)[lb]{\smash{{\SetFigFont{8}{9.6}{\rmdefault}{\mddefault}{\updefault}{\color[rgb]{0,0,0}s}%
}}}}
\put(4596,-47){\makebox(0,0)[lb]{\smash{{\SetFigFont{8}{9.6}{\rmdefault}{\mddefault}{\updefault}{\color[rgb]{0,0,0}t}%
}}}}
\put(3505,448){\makebox(0,0)[lb]{\smash{{\SetFigFont{8}{9.6}{\rmdefault}{\mddefault}{\updefault}{\color[rgb]{0,0,0}s'}%
}}}}
\end{picture}%
\end{center}
    \caption{The two cylinder construction of $\cyl^+(H)$ from Section~\ref{sec:even-cycle-with-at-least-one-vertex-loopless}. All vertices are loopless.}
    \label{fig:C6s}
\end{figure}

If $k\geq 6$ is even, let $\cyl^+(H)$ be two copies of $\cyl_{k+1}(H)$ identified on their bottom copies with a path $P_{(k/2)-1}$ of length $(k/2)-2$ added joining the vertices labelled $(k-1,0)$ in the two top copies of $H$ in $\cyl^+(H)$. One end  of this path is labelled $s$ and the other is labelled $s'$. Now, let $t$ be the vertex on the same cycle as $s'$ but directly opposite it. This gadget is drawn in Figure~\ref{fig:C6s}. Let $\Pi(k/2)$ be the subset of elements of $\{0,\ldots,k-1\}$ with the same parity as $k/2$.
\begin{lemma}
    Let $H$ be a cycle of even length $k \geq 6$ with at least one irreflexive vertex.  Let $f$ range over all homomorphisms from $\cyl(H)$ to $H$ so that $f(H_0)$ is the identity. Then the orbit of $(s,t)$ in $V(H)^2$ restricted to $\Pi(k/2)^2$ is precisely $s \neq t \in \Pi(k/2)$.
\label{lem:no-majority-even-not-reflexive}
\end{lemma}
\begin{proof}
The inductive argument is the same as in the previous lemma. We use the one irreflexive vertex to force that successive copies of the cycle cannot be mapped identically. It follows that $s,t$ must be mapped $\{0,2,\ldots,k-2\}$. Note that $s\neq t$ is enforced directly by the path joining their respective copies of $H$.
\end{proof}

\subsubsection{Fully reflexive cycle}

In these cases, our cylinder construction is very close to that in \cite{FederH98}. However, we will only have one vertex at level $i+1$ attached to each vertex on level $i$ (whereas they have two). The reason for this discrepancy is that their construction works for the reflexive $4$-cycle, whereas ours will not. 

If $k\geq 6$ is even, and $H$ is reflexive, let $\cyl^+(H)$ be $\cyl_k(H)$ with a pendant path $P_{(k/2)}$ of length $(k/2)-1$ added to vertex $(k-1,0)$ in the top copy of $H$ in $\cyl^+(H)$. Label the leaf of this path $t$ and label vertex $(k-1,k/2)$ as $s$.
\begin{lemma}
    Let $H$ be a reflexive cycle of even length $k \geq 6$. Let $f$ range over all homomorphisms from $\cyl^+(H)$ to $H$ so that $f(H_0)$ is the identity. Then the orbit of $(s,t)$ in $V(H)^2$ is precisely $s \neq t$.
\label{lem:no-majority-even-reflexive}
\end{lemma}
\begin{proof}
The argument is similar to that for Lemma~\ref{lem:no-majority-even-not-reflexive} except that now all $s \neq t \in \{0,\ldots,k-1\}$ is attainable.
\end{proof}
If $k\geq 5$ is odd, and $H$ is reflexive, let $\cyl^+(H)$ be $\cyl_k(H)$ with a new vertex $w$ made adjacent to both vertices $(k-1,(k-1)/2)$ and $(k-1,(k+1)/2)$. Now add a pendant path $P_{(k+1)/2-1}$ of length $(k+1)/2-2$ added to the vertex $w$. Label the leaf of this path $t$ and label vertex $(k-1,0)$ as $s$.
\begin{lemma}
    Let $H$ be a reflexive cycle of odd length $k \geq 5$. Let $f$ range over all homomorphisms from $\cyl^+(H)$ to $H$ so that $f(H_0)$ is the identity. Then the orbit of $(s,t)$ in $V(H)^2$ is precisely $s \neq t$.
\label{lem:no-majority-odd-reflexive}
\end{lemma}
\begin{proof}
The argument is similar to that for the previous lemma.
\end{proof}

\subsubsection{Bringing it all together}

\begin{lemma}
    Let $H$ be a forest of cacti containing a cycle of length $k\geq 5$ as an induced subgraph. Then $H$ is not preserved by an NU polymorphism of any arity.
    \label{lem:bringing-together}
\end{lemma}

\begin{proof}
To rule out a NU operation, we will argue that there is a definable pp-relation that is violated by that operation. 
Since an NU operation is idempotent we may assume, on the relational side for our pp-definitions, that we have access to constants \cite{barto2017polymorphisms}. Let $H'$ be some cycle in $H$. 
  
  If $H'$ is of odd length, or both even and reflexive, then the construction $\cyl_{k+1}(H)$ demonstrates that the vertices of the cycle are pp-definable from the constants and may be considered a subalgebra of size at least $5$.  Indeed, since $H$ is a cactus, the successive copies of $H'$ in $\cyl_{k+1}(H)$ cannot deviate from $H'$. We may therefore existentially quantify all variables except the one associated with $(k-1,0)$ to prove this (assuming constants on the base copy). However, since disequality is pp-definable on this subalgebra (according to Lemmas~\ref{lem:no-majority-odd-not-reflexive}, \ref{lem:no-majority-even-reflexive} and \ref{lem:no-majority-odd-reflexive}) we know all idempotent polymorphisms on this subalgebra are projections. This rules out the possibility of an NU polymorphism.

    If $H'$ is of even length and not reflexive, then existentially quantifying $s$ (see Figure~\ref{fig:C6s}) in $\cyl^+_{k}(H)$ shows that $\Pi(k/2)$ is a subalgebra of size at least $3$. Furthermore, disequality is pp-definable on this subalgebra, according to Lemma~\ref{lem:no-majority-even-not-reflexive}, so all idempotent polymorphisms on this subalgebra are projections. Again, this rules out the possibility of an NU polymorphism.
\end{proof}

\subsection{Small Cycles}
\subsubsection{Irreflexive $3$-cycle}

We make a similar argument for the irreflexive triangle as in the previous section. 
\begin{lemma}
    Let $H$ be an irreflexive triangle. In $\cyl_3(H)$, let $s,t$ be adjacent vertices in the top copy of $H_0=H$. Let $f$ range over all homomorphisms from $\cyl_3(H)$ to $H$ so that $f(H_0)$ is the identity. Then the orbit of $(s,t)$ in $V(H)^2$ is precisely $s \neq t$.
\label{lem:no-majority-K3}
\end{lemma}
\begin{proof}
The argument is similar to that for Lemma~\ref{lem:no-majority-odd-not-reflexive}.
\end{proof}
Note that the previous lemma fails if $H$ has any reflexive vertex.
\begin{lemma}
    Let $H$ be a forest of cacti containing an irreflexive triangle as an induced subgraph. Then $H$ is not preserved by an  NU polymorphism.
    \label{lem:bringing-together-bis}
\end{lemma}
\begin{proof}
The argument is similar to that for Lemma~\ref{lem:bringing-together}.
\end{proof}

\subsubsection{Square with at least three reflexive vertices }

 Note that in \cite{brewster2008near} the authors showed that no reflexive cycles of length at least four admit an NU polymorphism (without the use of complexity assumptions), we show that this is also the case for squares, i.e. cycles of length four,  with  exactly three reflexive vertices. 

\begin{lemma}
Let $H$ be a cactus containing a square with three reflexive vertices as a subgraph. Then $H$ is not preserved by an  NU polymorphism.
\end{lemma}

\begin{proof}
Let $H$ be a  cactus containing the square with three reflexive vertices as a subgraph.

Let $ \{(0, 1), (1, 0), (1, 2), (2, 1), (2, 3), (3, 2), (3, 0), (0, 3), (0, 0), (2, 2),(3, 3)\} \in E(H)$. 

For a contradiction, assume that $H$ is preserved by an NU operation,  $n$,  of some arity $k\ge 3$. Let $e_1=n(0, 0, \ldots, 0, 1, 2)$, 
Since $(0,1), (2,1), (1, 2)\in E(H)$ we have 
$$ n((0,1),(0, 1), \ldots, (0,1), (1, 2), (2, 1))=(n(0, \ldots, 0, 1, 2), n(1, \ldots, 1, 2, 1))=(e_1, 1)\in E(H)$$
In a similar way, using the edges $(0, 3), (1, 2), (2, 3)$ we obtain $(e_1, 3)\in E(H)$, and because $(0,0), (1, 0), (2, 1)\in E(H)$ we get
$$  n((0,0), \ldots, (0,0), (1, 0), (2, 1))=(n(0, \ldots, 0, 1, 2), n(0, \ldots, 0, 0, 1))=(e_1, 0)\in E(H).$$

Since $H$ is a  cactus and $(e_1, 1), (e_1, 3), (e_1, 0)\in H$ we must have that $e_1=0$, as otherwise we would have two cycles sharing the edge $(1, 0)$ (for example).

Now, let $e_2=n(0, \ldots, 0,  1 , 2 ,2)$, $e_3=n(0, \ldots, 0, 1, 2, 2, 2), \ldots, e_{k-3}=n(0, 0, 1, 2, \ldots, 2), e_{k-2}=n(0, 1, 2, \ldots, 2)$.
Since $n$ is edge-preserving and $(0, 0), (0, 1), (1, 2),(2, 2)\in E(H)$, we have that 
 $(e_i,e_{i+1}) \in E(H)$ for all $i=1, \ldots, k-3 $. 

As we did above for $e_1$, we can show  that $(e_i, 1), (e_i, 3)\in E(H)$ for all $i=1, \ldots, k-2$.
It then follows that  $(e_2, 0), (e_2, 1), (e_2, 3)\in E(H)$, and as $H$ is a cactus we must have $e_2=0$. Similarly we must have $e_3=\cdots =e_{k-2}=0$.
This yields $n(0,1,2, \ldots 2)=0$.
Because $(1,2), (2, 2)\in E(H)$ we have that $$n((0, 0),(1, 2), (2, 2), \ldots, (2, 2))=(n(0, 1, 2, \ldots, 2), n(0, 2, \ldots, 2))=(0, 2) \in E(H),$$
a contradiction.

    \end{proof}

\section{Cacti with majority}\label{maj}

In this section we construct an edge-preserving majority for forests of cacti without the induced subgraphs mentioned in Theorem~\ref{theomajority}. We first construct a majority for a connected cactus without the forbidden induced subgraphs, and then show how the majority can be extended to forests of cacti.

Let $G=(V(G), E(G))$ be a cactus graph that is loop-connected. For each two vertices $u, v \in V(G)$ we define the {\it distance} $d(u, v)$ as the length of a shortest path connecting $u$ to $v$.
To restrict ourselves to cacti without the forbidden induced subgraphs described in Theorem~\ref{theomajority},  we assume that all cycles in $G$ are triangles with at least one reflexive vertex, and squares with at most two reflexive vertices.  Let $C_1, C_2, \ldots, C_p$ be all cycles of $G$, where $C_1, \ldots, C_q$ for some $q\le p$, are the squares with two reflexive vertices. Assume that at least one of these exists, for ease of index notation, if they do not then the following still holds ignoring the relevant parts of the majority. 
 
For each cycle  $C_i$ label its vertices with  $0_i, 1_i, 2_i$, and $3_i$ if $C_i$ is a square, anticlockwise as follows: 
\begin{itemize}
\item when $C_i$ is irreflexive, $0_i$ is the vertex closest (i.e. at shortest distance) to the reflexive component of $G$;
\item when $C_i$ has exactly one reflexive vertex, that vertex should be labelled $0_i$;
\item when $C_i$ has exactly two reflexive vertices they are labelled $0_i, 2_i$ if $C_i$ is a triangle, and $0_i, 3_i$ if $C_i$ is a square. Note that, in the latter case the edges of $C_i$ are $(0_i, 1_i), (1_i, 2_i), (2_i, 3_i)$ and $(3_i, 0_i)$;
\item when $C_i$ is a fully reflexive triangle, starting at an arbitrary vertex.
\end{itemize}

Note that a vertex belonging to two cycles $C_i$ and $C_j$ may be labelled in two ways i.e. $0_i$ and $2_j$, but this shall not cause confusion in what follows.

Let $G_T$ be the graph obtained from $G$ by removing exactly one edge from each cycle as specified below.
$G_T$ is a spanning graph that is a tree in the irreflexive components of $G$. It is obtained by removing from $C_i$ the edge 
\begin{itemize}
\item $(2_i, 3_i)$ when $C_i$ is a square with two reflexive vertices; 
\item $ (3_i, 0_i)$ when $C_i$ is a square with at most one reflexive vertex;
\item $(1_i, 2_i)$ when $C_i$ is a triangle.
\end{itemize}

Let $A$ be the reflexive component of $G$,  and denote by 
 $X_1, X_2, \ldots, X_k$, for some $k\ge q \ge 0$
 ,  the irreflexive components  of $G_T$, with $X_1, \ldots, X_q$ containing the irreflexive parts of the squares $C_1, \ldots, C_q$ (the squares with two reflexive vertices). We label each reflexive vertex that neighbours an irreflexive vertex in $G_T$, 
$p_{X_i}$, where $X_i$ is the connected irreflexive subtree given by removing the vertex $p_{X_i}$. Note that we can have $p_{X_i}=p_{X_j}$, for many $j's$ as a reflexive vertex might share edges with many irreflexive vertices.

     In each $X_i$, we label 
each leaf $l_1,...,l_{n_i}$ such that  $d(l_{1}, p_{X_i})\geq d(l_2, p_{X_i})\geq \cdots \geq d(l_{n_i},p_{X_i})$. Here and in what follows, unless mentioned otherwise, we take distances in $G_T$, i.e. considering  paths in $G_T$, the spanning graph of $G$. 
For each $j=1, \ldots, n_i$ in each $X_i$, we call the  path from $l_j$ to $p_{X_i}$, $P_j$; and for each $y\in X_i$ we call the unique vertex on $P_j$ that is of shortest distance to $y$, $k^j_y$ (note that if $y \in P_j$ then $k^j_y=y$).
It is worth remarking that as leaves and paths are labelled within the separate irreflexive components,  we may have leaves and paths with the same labelling, this shall not cause any confusion in the following, as we will be using the definitions only within one irreflexive component.

We colour the irreflexive components of $G_T$ with two colours as follows:  for all $i=1, \ldots, k$ the irreflexive neighbours of $p_{X_i}$ are coloured with the same  colour, 
and propagate the colouring through all vertices of the corresponding $X_i$, in such a way that adjacent vertices do not share the same colour. Given $a, b \in \bigcup_{i=1}^k X_i$ we write $[a] = [b]$ if $a$ and $b$ were assigned the same colour, and $[a] \neq [b]$ otherwise. We call these `colour classes'.

For each vertex $z\in X_i$, for any $i=1, \ldots, k, $ we let $z^+\in X_i\cup\{p_{X_i}\}$  be the neighbour, in $G_T$  of $z$ closest to $p_{X_i}$. 

For any three vertices $a, b, c\in V(G)$ their {\it{median}}, $med\{a,b,c\}$, is a vertex that lies on a shortest path between each pair of vertices $a,b,c$. In $G$, the median is not always defined, 
to overcome this, and to make the majority function that we will define edge-preserving, when referring to $x=med\{a, b, c\}$ we mean the median in the subgraph of $G_{S}$ obtained by adding to the spanning graph $G_T$ the edges $(3_i, 2_i)$ for all $i=1, \ldots, q$, and removing any edge from the reflexive triangles. So $G_{S}$ will be a {\it{spanning cactus}} whose cycles are squares with exactly two reflexive vertices. It is then easy to check that in $G_S$ the median is always defined, and is always unique (despite the shortest paths between vertices not necessarily being unique). 

For each triple of vertices  $a, b, c\in V(G)$ (not necessarily all distinct) we define a new triple 
of vertices $u, v,w$, such that $\{u, v, w\}=\{a, b, c\}$ as a multiset, as follows:
\begin{enumerate}
\item[Case 0.] if $med\{a, b, c\}\in A$ then $u, v, w$ are assigned in any order, just ensuring that if any of them belongs to an irreflexive component then $u$ should be one of the irreflexive vertices.
\end{enumerate}

\noindent When at least two of $a, b,c$  belong to the same irreflexive connected component, $X_i$ for some $i=1, \ldots, k$, we do the following:
\begin{enumerate}

\item[Case 1.] if  $a, b, c\in P_{t}\subseteq X_i$ for some  $i=1, \ldots, k$  and some $t=1, \ldots, n_i$   we  define $u, v, w$ such that $d(u, p_{X_i})\geq d(v, p_{X_i})\geq d(w, p_{X_i})$;
        
\item[Case 2.] if $a, b, c\in  X_i$ for some $i=1, \ldots, k$, but only two of $a, b, c$ belong to the same path, we choose $P_t$,  $t=1, \ldots n_i$, to be the smallest possible index path containing two of $a, b,c$, and define $u, v, w$ such that  $d(k_u^t, p_{X_i})\geq d(k_v^t, p_{X_i})\geq d(k_w^t, p_{X_i})$. 
Note that we will always have $u\in P_t$ so $k_u^t=u$.

\begin{figure}[H]
\centering
\resizebox{0.6\textwidth}{!}{%
\begin{circuitikz}
\tikzstyle{every node}=[font=\Large]
\node at (1.25,15.5) [circ] {};
\node at (6.25,15.5) [circ] {};
\node at (11,15.5) [circ] {};
\node at (15,15.5) [circ] {};
\node at (4,15.5) [circ] {};
\node at (6.25,16.5) [circ] {};
\node at (6.25,18) [circ] {};
\draw (1.25,15.5) to[short] (15,15.5);
\draw (6.25,18) to[short] (6.25,15.5);
\node [font=\Large] at (6.75,16.5) {$a$};
\node [font=\LARGE] at (7,18.5) {$P_2$};
\node [font=\LARGE] at (1,16.5) {$P_1$};
\node [font=\Large] at (4,14.75) {$b=k^1_b$};
\node [font=\Large] at (6.25,14.75) {$k^1_a$};
\node [font=\Large] at (11,14.75) {$c=k^1_c$};
\draw [->, >=Stealth] (1.5,16.5) -- (2.75,16.5);
\draw [->, >=Stealth] (7,18) arc (180:270:2cm);
\end{circuitikz}
}%
\caption{Case 2. with $u=b$, $v=a$, and $w=c$. }
\label{fig:casetwo}
\end{figure}
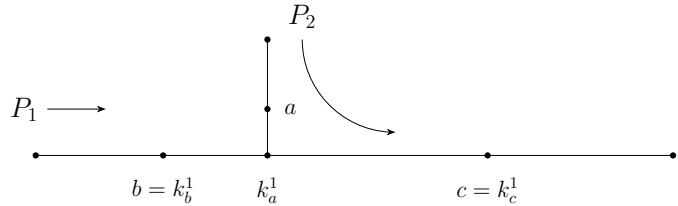


\item[Case 3.] if $a, b, c\in  X_i$ for some $i=1, \ldots, k$, are all three on different paths, say $P_j$,$P_k$,$P_m$, then we take  $t=min \{j,k,m\}$, and let  $u,v,w$ be such that $d(k^t_u,p_{X_i}) \geq d(k^t_v,p_{X_i}) \geq d(k^t_w,p_{X_i})$. If two of $k^t_u$, $k^t_v$, $k^t_w$ are equal, then we take the next lowest indexed path to base this on, i.e., $P_s$ with $s=min\{\{j,k,m\} \,\backslash\, t\}$; if this distinction is not possible, i.e. all indices result in two of $k$s being equal, then we take $u$ and $v$ to be two of the vertices, of $a, b, c$,  in the same colour class. 
\begin{figure}[H]

\centering
\begin{subfigure}[h]{0.4\linewidth}
\resizebox{1.05\textwidth}{!}{%
\begin{circuitikz}
\tikzstyle{every node}=[font=\Large]
\node at (1.25,15.5) [circ] {};
\node at (6.25,15.5) [circ] {};
\node at (11,15.5) [circ] {};
\node at (15,15.5) [circ] {};
\node at (4,15.5) [circ] {};
\node at (6.25,16.5) [circ] {};
\node at (6.25,18) [circ] {};
\draw (1.25,15.5) to[short] (15,15.5);
\draw (6.25,18) to[short] (6.25,15.5);
\node [font=\Large] at (6.75,16.5) {$a$};
\node [font=\LARGE] at (7,18.5) {$P_2$};
\node [font=\LARGE] at (1,16.5) {$P_1$};
\node [font=\Large] at (4,14.75) {$b=k^1_b$};
\node [font=\Large] at (6.25,14.75) {$k^1_a$};
\node [font=\Large] at (11,16.25) {$k^1_c$};
\draw [->, >=Stealth] (1.5,16.5) -- (2.75,16.5);
\draw [->, >=Stealth] (7,18) arc (180:270:2cm);
\draw (11,13) to[short] (11, 15.5);
\node [font=\Large] at (11.5,14.5) {$c$};
\node at (11,14.5) [circ] {};
\node at (11,13) [circ] {};
\draw [->, >=Stealth] (11.75,13) arc (180:90:2cm);
\node [font=\LARGE] at (11.75,12.5) {$P_3$};
\end{circuitikz}
}%
\caption{Distinct $k^1$'s.}
\label{fig:casethreept1}
\end{subfigure}%
\hspace{0.5cm}
\begin{subfigure}[h]{0.35\linewidth}
\resizebox{\textwidth}{!}{
\begin{circuitikz}
\tikzstyle{every node}=[font=\Large]
\node at (1,15) [circ] {};
\node at (4,15) [circ] {};
\node at (6,15) [circ] {};
\node at (10,15) [circ] {};
\node at (6,16) [circ] {};
\node at (6,19) [circ] {};
\node at (4,19) [circ]{};

\node at (4.5,18.25)[circ]{};

\node at (6,18.25)[circ]{};
\node [font=\LARGE] at (0.75,16) {$P_1$};
\node [font= \LARGE] at (3,19.5){$P_2$};
\node [font= \LARGE] at (6.75,19.5){$P_3$};

\draw (1,15) to [short] (10,15);
\draw (6,15) to [short] (6,19);

\draw (4,19) to [short] (6,16);

\draw [->, >=Stealth] (6.75,19) arc (180:270:1.75cm);
\draw [->, >=Stealth] (1.25,16)--(2.5,16);
\draw [->, >=Stealth] (3.25,19)--(5.25,16);

\node [font=\Large] at (5,18.25) {$a$};
\node [font=\Large] at (6.5,18.25) {$b$};
\node [font=\Large] at (4,14.5) {$c=k^1_c$};
\node [font=\Large] at (7,14.5) {$z=k^1_a=k^1_b$};

\end{circuitikz}
}
\caption{Non distinct $k^1$'s, distinct $k^2$'s}
\label{fig:casethreept2}
\end{subfigure}%
\hspace{1cm}
\begin{subfigure}[h]{0.4\linewidth}
\resizebox{\textwidth}{!}{
\begin{circuitikz}
\tikzstyle{every node}=[font=\Large]
\node at (1,15) [circ] {};
\node at (4,15) [circ] {};
\node at (6,15) [circ] {};
\node at (8,15) [circ] {};
\node at (12,15) [circ] {};
\node at (6,16) [circ] {};
\node at (6,17) [circ] {};
\node at (6,14) [circ] {};
\node at (6,13) [circ]{};

\node [font=\LARGE] at (0.75,16) {$P_1$};
\node [font= \LARGE] at (6.75,17.5){$P_2$};
\node [font= \LARGE] at (6.75,12.5){$P_3$};

\draw (1,15) to [short] (12,15);
\draw (6,13) to [short] (6,17);

\draw [->, >=Stealth] (6.75,17) arc (180:270:1.75cm);
\draw [->, >=Stealth] (6.75,13) arc (180:90:1.75cm);
\draw [->, >=Stealth] (1.25,16)--(2.5,16);

\node [font=\Large] at (6.5,16) {$a$};
\node [font=\Large] at (6.5,14) {$b$};
\node [font=\Large] at (4,14.5) {$c$};
\node [font=\Large] at (5.75,15.25) {$z$};

\end{circuitikz}
}
\caption{Non distinct $k^1$'s, $k^2$'s and $k^3$'s.}
\label{fig:casethreept3}
\end{subfigure}%
\caption{Three variants on Case 3.}
\subcaption*{In Figure \ref{fig:casethreept1}, we take $u=b$, $v=a$, and $w=c$.
In Figure \ref{fig:casethreept2}, we have non distinct $k^1$'s, and so we take the path ordering based on $P_2$, hence, $u=a$, $v=b$ and $w=c$.
In Figure \ref{fig:casethreept3}, we can see that $k^1_a=k^1_b=z$, $k^2_b=k^2_c=z$, and $k^3_a=k^3_c=z$, and so we take $u$ and $v$ to be two of the three vertices that are in the same colour class. }
\end{figure}
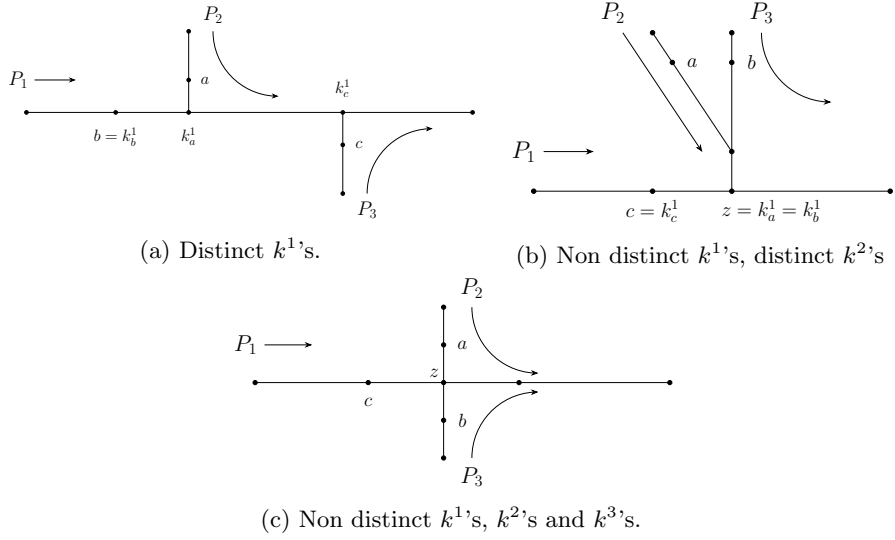



 \item[Case 4.] if only two of  $a, b, c$ belong to $X_i$, for some $i=1, \ldots, k$  and the other element does not belong to $X_i$, we then define $u, v, w$ such that $u, v\in \{a, b, c\}\cap X_i$, $d(u, p_{X_i})\ge d(v, p_{X_i})$ and  $w\not\in  X_i$. 

\end{enumerate}

To simplify notation in the definition of the majority function, we let $B$  be the set containing all triples $(a, b,c)$ such that: two of $a, b,c$ equal   $0_i$ and $3_i$ for some $i=1, \ldots, q$, and the third one belongs to $X_i$,  for the same $i$. 

Recall that $X_1, \ldots, X_q$ are the irreflexive components associated with the squares with two reflexive vertices, 
 note that when $(a, b, c)\in B$ we are in Case 0. above, and  $\{u, v, w\}=\{a, b, c\}$ are such that $\{v, w\}=\{0_i, 3_i\}$ and $u\in X_i$.

We are now in a position to, finally, define a majority function $m:G^3\longrightarrow G$
as follows: $m$ is symmetric and 

\begin{equation}
    m(a,b,c)=\begin{cases}
    0_i,& {\rm if}\ (a, b, c)\in B,  \  u \in X_i, [u]=[2_i];\  \\
    3_i,& {\rm if}\ (a,b, c)\in B,\   u\in X_i, [u]=[1_i]; \\
x, & \text{if\ }(a,b, c)\notin B, \text{\ and\ } x\in A \ \text{\ or \ } [u]=[v]=[x]; \\
x^+,& \text{if \ } [u]=[v]\neq[x];  \\
p_{X_i}, & \text{if \ } [u]\neq [v],\  u, v\in X_i, \text{\ and\ } w\not\in X_i,\  i=q+1, \ldots, k; \\
0_i ,& \text{if \ }  [u]\neq[v],  d(w,0_i)<d(w,3_i),  \ u, v\in X_i,\ \text{and\ }  w\notin X_i, \ i=1, \ldots, q;\\
3_i ,& \text{if \ }[u]\neq [v],  d(w,3_i)<d(w,0_i),\ u, v\in X_
i,\ \text{and\ } w\notin X_i, \ i=1, \ldots, q;\\
k^t_w, & \text{if \ } [u]\neq [v] \text{\ and\ } [w]=[k^t_w]; \\
(k^t_w)^+, & \text{if \ } [u]\neq [v] \text{\ and\ } [w]\neq[k^t_w], \\
\end{cases}
\label{majorityfunction}
\end{equation}
with $x=med\{a, b, c\}$. 

We note that, if for any $i=1, \ldots, k$ at most one of $a, b,c$ belongs to $X_i$ then we must have $x= med\{a, b, c\} \in A$, or in other words, the only cases where the median is not reflexive is when at least two of the vertices are in the same irreflexive component, so in the definition of $m$ all lines, except the first two,  apply to cases when  $u, v \in X_i$, for some $i=1, \ldots, k$. It is then easy to see that, using   Cases 1.-4. above, that  $m$ is defined for all $a, b, c\in V(G)$.

\begin{prop}
 The function $m$ is a majority function.
 \end{prop}
 \begin{proof}

Let $a, b\in V(G)$ be arbitrary, as $m$ is symmetric by definition, it is enough for us to show that $m(a, a, b)=a$.
We have $med\{a, a, b\}=a$. If $a\in A$ then we have $m(a, a, b)= med\{a, a, b\}=a$.
If $a\notin A$ then the vertices are in one of Cases 1., 2.  or 4., and  defining $u, v, w$ as described there we get $u=v=a$ or $v=w=a$. When  $[u]=[v]$  we obtain $m(a, a, b)= med\{a, a, b\}=a$. Suppose now that $[u]\neq [v]$, then we must have $u=b$ and $v=w=a$ and be in Case 1. 
Then we will have 
$$m(a, a, b)=m(u, v, w)=w=a.$$

When $a=b$ we get that $m$ is idempotent. 
\end{proof}

\begin{prop}\label{proofmaj}
The majority function $m$  preserves the edges of $G$, i.e is a polymorphism of the cactus $G$.
\end{prop}

This will be proved in the next subsection.

\subsection{Proof of Proposition~\ref{proofmaj}}

Let  $(a, a'), (b, b'), (c, c')\in E(G)$ be arbitrary edges of the cactus graph $G$, and let $m(a, b, c)=z$ and $m(a', b', c')=z'$. We want to show that $(z, z')\in E(G)$. We use all the partitions and definitions in $G$ as at the start of this section, and look at all of the cases (0. through to 4.) one at a time.

\subsubsection{Case 0.}\label{case0}
 Suppose first that, for each $i=1, \ldots, k$, at most one of $a,b,c$ belong to $X_i$.  
Then we have $x=med\{a,b,c\} \in A$, and  either $(a, b, c)\in B$ and $z\in \{0_i, 3_i\}$ for some $i=1, \ldots, q$, or  $z=med\{a,b,c\}=x$.

Assume first that $(a, b, c)\notin B$, and $z=med\{a,b,c\}=x$. Furthermore start by assuming that  $a,b,c$ lay all on the same  path, then without loss of generality we can assume that $med\{a, b, c\}=a$, and so $z=a$.
If also $a',b',c'$ lay on a path, in $G_T$, then one of the following cases holds:
\begin{enumerate}
\item $med\{a', b', c'\}=a'$; in this case if $a'\in A$ we have $z'=a'$, yielding $(z, z')\in E(G)$.
When $a'\notin A$  
we must have $a=p_{X_i}$ and $a'\in X_i$, for some $i=1,\ldots, k$. If none of 
$b', c'$ belongs to $X_i$ then  $a', b', c'$ do not lay on the same path, as $med\{a', b', c'\}=a'$,  so at least one of $b', c'$ must also belong to $X_i$.

If  exactly  one of $b', c'$ also belongs to $X_i$, without loss of generality  assume that it is $b'$, so $b, b'\in X_i$. If $[a']=[b']$ then $z'=a'$, and if $[a']\neq [b']$ then $z'=p_{X_i}=a$.

If $b', c'\in X_i$ then we must have that $a, b, a', b'$ belong to a triangle $C_j$ with vertices $0_j, 1_j, 2_j$, and  exactly one reflexive vertex $a=0_j$.  We can  then have $b=1_j$ (if $b=2_j$ the argument is symmetric) and $a'=2_j=b'$ and  $c'$ belonging to the irreflexive component containing $b'=2_j$, which implies that  $z'=m(2_j, 2_j, c')=2_j$.
Thus we always get $(z, z')\in E(G)$.

\item $med\{a', b', c'\}=a$, which happens when $(a, b)\in E(G)$  and $a'=b, b'=a$ (or equivalently $a'=c$, $c'=a$)
; or when $a=p_{X_i}$ and $a'\in X_i$, for some $i=1,\ldots, k$ such that $p_{X_i}=0_j$, the reflexive vertex in a triangle with one reflexive vertex, $C_j$. 
\noindent In both cases, $a \in A$ so $m(a',b',c')=a$, and $z'=z=a$.

\item  $med\{a', b', c'\}=b'\notin A$, (if $b' \in A$ then $b'=a$, and we have considered this above) which happens when $a, a', b, b'$ belong to a triangle $C_j$ with vertices $0_j, 1_j, 2_j$, and  exactly one reflexive vertex $a=0_j$. Recall that $1_j$ and $2_j$ belong to different irreflexive components here. We can then have $a'=1_j=b$ and  $c'$ belonging to the irreflexive component containing $b'=2_j$, which implies that  $z=a=0_j$, $med\{1_j, 2_j, c'\}=2_j$ and 
$$z'=m(a', b', c')=m(1_j, 2_j, c')= \left\{\begin{array}{cl}
2_j, & {\rm if} \ [2_j]=[c'];\\
(2_j)^+= 0_j, & {\rm if} \ [2_j]\neq [c'],
\end{array}\right.$$
so we always get $(z, z')\in E(G)$.

\end{enumerate}

If $a',b',c'$ do not lay on a path, we must have  moved off the $a, b, c$ path via $x$, which implies that $med\{a',b',c'\}=x\in A$ and so  $z'=m(a',b',c')=x$, thus $(z, z')\in E(T).$

Assume now that $a,b,c$ do not lay on the same path. 
Then $med\{a,b,c\}=x \notin \{a,b,c\}$. If $a',b',c'$ continue to not lay on a path, then $med\{a',b',c'\}=x \in A$, hence $z'=m(a',b',c')=x$. If $a',b',c'$ join onto a path, they do so via the median and hence, $med\{a',b',c'\}=x$ again, giving us that $z'=m(a',b',c')=x$. Hence in all cases $(z, z')\in E(T)$.

We now look at the case when  $(a, b, c)\in B$. Then there exists $i=1, \ldots q$ such that two of $a,b,c$ equal $0_i, 3_i$ and the third one, which we call $u$, belongs to $X_i$.  This implies that $z=0_i$ when $[u]=[2_i]$ and $z=3_i$ when $[u]=[1_i]$.

Without loss of generality, assume that $a=0_i, b=3_i$ and $c\in X_i$. Suppose that $[c]=[2_i]$ (we will look at the case when $[c]=[1_i]$ in brackets as the differences are minimal).
We have $z=m(0_i, 3_i, c)=0_i$ ($z=m(0_i, 3_i, c)=3_i$), and $z'=m(0_i', 3_i', c')$.
If at least two of $0_i', 3_i', c'$ are any combination of $0_i$'s and $3_i$'s, we get $z'\in \{0_i,3_i\}$, yielding $(z, z')\in E(G)$.
The same happens if $0_i',3_i'\in V(G)\backslash X_i$ as $z'$ would be the median of all vertices $0_i', 3_i', c'$. The only cases left to consider are  when $0_i'$, $3_i'$ or both are in $X_i$: 
\begin{itemize}
\item If  $0_i'=1_i$ and $3_i'=2_i$ then either $c'=3_i$ ($c'=0_i$) or there exists a path in $G_T$ from $0_i$ to $c'$ containing at least $1_i$ (it may also contain $2_i$).

In the first case, as $[1_i]\neq[2_i]$, and $d(3_i,3_i)<d(3_i,0_i)$, $m(0_i',3_i',c')=3_i$ (as $[1_i]\neq[2_i]$, and $d(0_i,0_i)<d(0_i,3_i)$, $m(0_i',3_i',c')=0_i$). Hence $(z,z')\in E(G)$.

Otherwise, we are either in Case 1. or 2., and defining $\{u, v, w\}=\{0_i', 3_i', c'\}$ accordingly, we get that $w=1_i$, so we have that $[u]\neq[v]$, as $[c']\neq [2_i]$ ($[u]=[v]$, as $[c']=[2_i]$). Hence we get  $z'=m(1_i, c',2_i )=1_i$ which is adjacent to $0_i$ ($z'=m(1_i, c', 2_i)= 2_i$  which is adjacent to $z=3_i$).

\item     If $0_i'=1_i$, $3_i'\neq 2_i$, and $c' \in X_i$ ($c' \notin X_i$ has been considered above), then we have that two vertices are in $X_i$. As $[1_i]=[c']$, we get that $z'=m(1_i,c',3_i')\in \{x'=med\{1_i,c',3_i'\}, (x')^+$\} ($[c']\neq [1_i]$ so $m(1_i,c',3_i')\in \{0_i, 3_i\}$ which are both adjacent to $3_i$).  If $x'=1_i$ then $z'=1_i$, and if $x'=2_i$ then we get $z'=2_i^+=1_i$  so $(z, z')\in E(G)$. 

\item   If $0_i'\neq1_i$, $3_i'= 2_i$, and  $c' \in X_i$, then as $[2_i]\neq[c']$ , we get that $z'=m(0_i',c',2_i)\in \{0_i, 3_i\}$ depending if $0_i'$ is closer to $0_i$ or $3_i$. In either case $(z, z')\in E(G)$.

     (If $c' \in X_i$, $0_i'\neq1_i$ and $3_i'= 2_i$, then as $[c']=[2_i]$ we get $[u]=[v]$ and  $x'=med \{0_i', c', 2_i\}\in \{1_i, 2_i\}$. If $x'=1_i$ then $z'=1_i^+=0_i$ and if $x'=2_i$ we get $z'=2_i$. In both cases we get $(z, z')\in E(G)$.)
\end{itemize}

\subsubsection{Case 4.}\label{case4}
  
  Assume that exactly two of $a,b,c$ belong to  $X_i$, for some $i=1, \ldots, k$. Let $u,v \in \{a,b,c\} \cap X_i$ be such that $d(u, p_{X_i})\geq d(v, p_{X_i})$,  and $w\in\{a, b, c\}$ be such that $w\not\in X_i$.  

We then know that $$z=m(u,v,w)=\begin{cases}
    x, & \text{if } [u]=[v]=[x];\\
    x^+, & \text{if } [u]=[v]\neq[x];\\
    p_{X_i}, & \text{if \ } [u]\neq [v] \text{\ and\ }  i=q+1, \ldots, k; \\
0_i ,& \text{if \ }  [u]\neq[v],  d(w,0_i)<d(w,3_i),  \ \text{and\ }  i=1, \ldots, q;\\
3_i ,& \text{if \ }[u]\neq [v],  d(w,3_i)<d(w,0_i),\ \text{and\ } i=1, \ldots, q,\\
    \end{cases}$$
where $x =med\{a,b,c\} \in X_i$.

Note that if $ med\{a', b', c'\} \in A$ then we are in Case 0. and  the proof is symmetric to
what we did in Subsection~\ref{case0}.
 
So we can assume that: 
\begin{itemize}
    \item $med\{a', b', c'\}\in X_i$, and so at least two of $a', b', c'$ belong to $X_i$;  or  
    \item $med\{a', b', c'\}\in X_l$ for some $l\neq i$, with $X_i,X_l$ the two irreflexive components of a triangle with exactly one reflexive vertex.
\end{itemize}

In the latter case, letting $C_j$ be the triangle with reflexive vertex $0_j$ and irreflexive vertices $1_j, 2_j$, assuming that $X_i$ is the irreflexive component containing $1_j$ we must have that at least one of $a, b, c$ is equal to $1_j$.

If two (say $a$ and $b$) are equal to $1_j$, and $med\{a',b',c'\}=2_j$, then we get $m(a,b,c)=1_j$ and $m(a',b',c')\in \{0_j,2_j\}$ dependent on whether $[c']=[2_j]$ or not. Hence $(z,z')\in E(G)$.

If only one vertex (say $a$) is equal to $1_j$, then $m(a,b,c)\in \{1_j,0_j\}$, dependent on whether the other vertex in $X_i$ (say $b$) is in the same colour class as $a$. As $med\{a',b',c'\} =2_j$, we know that $c' \in X_l, $ so $m(a',b',c')
\in \{2_j,0_j\}$, dependent on whether $[c']=[2_j]$ or not. In all cases we get $(z, z')\in E(G)$.

We now assume that $med\{a', b', c'\}\in X_i$, and so at least two of $a', b', c'$ belong to $X_i$.

 Let $u',v',w' \in \{a',b',c'\}$ be such that  $(u, u'),(v,v'),(w,w')\in E(T)$, and let  $\bar{u},\bar{v},\bar{w} \in \{u',v',w'\}$ be the vertices 
 such that the ordering with respect to the appropriate indexed path $P_{\alpha}$ is $d(k_{\bar{u}}^{\alpha}, p_{X_i})\geq d(k_{\bar{v}}^{\alpha}, p_{X_i}) \geq d(k_{\bar{w}}^{\alpha}, p_{X_i})$ if all $\bar{u}, \bar{v}, \bar{w}\in X_i$, or  if $\bar{w}\notin X_i$ 
 the other two satisfy $d(\bar{u}, p_{X_i})\geq d(\bar{v}, p_{X_i})$. Recall that we consider distances in $G_T$.

Start by assuming that $\bar{u}=u'$ and $\bar{v}=v'$, or $\bar{u}=v'$ and $\bar{v}=u'$. Which implies that $\bar{w}=w'$. We look at all possibilities for the value of $z$:

\noindent\boxed{{\rm Case}\ 4.1.} Suppose first that $z=x=med\{u, v, w\}$, so that we have $[u]=[v]=[x]$.

If $x \in \{a,b,c\}$, then $x=v$ and $\exists t=1, \ldots, n_i$ such that  $u,v\in P_t$. 
If $\exists t'=1, \ldots, n_i$ such that  $\bar{u},\bar{v}\in P_{t'}$ then $med\{a', b', c'\}= \bar{v}$, as $[u]=[v]$ and so $[\bar{u}]=[\bar{v}]$, thus $z'=\bar{v}$ implying that $(z, z')\in E(T)$.  
If $\bar{u}$ and $\bar{v}$ belong to different paths in $X_i$, then we must have $med\{a',b',c'\}=v$. Since $[\bar{u}]=[\bar{v}]\neq[v]$, we have that $z'= m(a',b',c')=v^+$, yielding $(z, z')\in E(T)$.

If  $x \notin \{a,b,c\}$ then $u,v,w$ are not on the same path. We have then that $[u]=[v]=[x]$. 
If $u',v',w'$ are also not on the same path, then we have $med\{u', v', w'\}=x=med\{u, v, w\}$ and $[\bar{u}]=[\bar{v}]\neq[x]$. It follows that $z'=m(a',b',c')=x^+$, and consequently $(z, z')\in E(T)$.
Note that if  $u',v',w'$ belong to  the same path, then we must have $\bar{v}=x$, but $[v]=[x]$, a contradiction.

\noindent \boxed{{\rm Case}\ 4.2.} Assume that $z=x^+$, so that we have $[u]=[v]\neq[x]$.

 We then know that $u$, $v$ and $w$ do not lay on a path. Hence, if $a',b',c'$ are still not on a path then $x=med\{a', b', c'\}$, and $[u']=[v']=[x]$.
If $u', v', w'$ belong to the same path, the vertices ``moved'' via the median. We know then that $v'=x$, and so $[u']=[v']=[x]$. In both cases we have  $z'=m(a',b',c')=x$, and so  $(z, z')\in E(T)$.

\noindent \boxed{{\rm Case}\ 4.3.} Assume  that $z=p_{X_i}$, for some $i=q+1, \ldots, k$, so we have  $[u]\neq [v]$.
We then must have  $[\bar{u}]\neq [\bar{v}]$ and so $z'=p_{X_i}$ if $\bar{w}\notin X_i$, and $z'=\bar{w}$ when $\bar{w}\in X_i$.
In the latter case we must have $w=p_{X_i}$, thus we always have $(z, z')\in E(G)$.

\noindent \boxed{{\rm Case}\ 4.4.} Finally, assume that $z=0_i$ for some $i=1, \ldots, q$, so we have $u, v$ in an irreflexive component that contains the irreflexive vertices of a square with exactly two reflexive vertices $0_i, 3_i$. Also, we have $[u]\neq[v]$ and   $d(w,0_i)<d(w,3_i)$.
Then we have $[\bar{u}]\neq [\bar{v}]$ and we obtain  $z'\in \{0_i, 3_i\} $ when $\bar{w}\notin X_i$ (note that we get $z'=3_i$ only when $w=0_i$ and $\bar{w}=3_i$), and $z'=\bar{w}$ when $\bar{w}\in X_i$. The latter case occurs when $w=0_i$, since we are assuming that $d(w, 0_i)< d(w, 3_i)$ we cannot have $w=3_i$. Thus $\bar{w}=1_i$ and we have $(z, z')\in E(G)$.

\noindent The case when $z=3_i$ can be solved in a similar way.

\bigskip

Suppose now that $\bar{u}, \bar{v}, \bar{w}$ do not keep the same ``order'' as $u', v', w'$. We have the following possibilities:
\begin{enumerate}
\item $\bar{u}=u', \bar{v}=w'=v$ and $\bar{w}=v'=w$, with $(v, w)\in E(G)$. As $w\notin X_i$ this can only happen when $w\in \{0_i, 3_i, p_{X_i}\}$, and in this case, we have $z=v$ if $[u]=[v]$ and $z=w$ if $[u]\neq [v]$.  We then get 
$z'=m(u', v, w)=v$ when $[u']=[v]$ and $z'=w$ if $[u']\neq [v]$. In both cases we get $(z, z')\in E(G)$.
\end{enumerate}
\noindent When $v, w $ are the vertices of a triangle with exactly one reflexive vertex $0_j$ and irreflexive vertices $1_j, 2_j$ we have the following extra case, recalling that $med\{a',b',c'\}\in X_i$. Without loss of generality, let  $X_i$ be the irreflexive component containing $1_j$, and recall that $2_j\notin X_i$. Let $X_{l}$ be the irreflexive component containing $2_j$.

\begin{enumerate}\setcounter{enumi}{1}
\item $\bar{w}=v', \bar{v}=w'$ and $\bar{u}=u'$, with $u\in X_i, v=1_j, w=2_j$ and $v'\in \{0_j,2_j\}, w'=1_j$.
When $[u]=[v]$ we get $z=med\{u, 1_j, 2_j\}=1_j$ and, because $[\bar{u}]\neq [\bar{v}]$ we get $z'=m(v',1_j, u')=0_j$.
When $[u]\neq [v]$ we get $z=0_j$ and as $[\bar{u}]=[\bar{v}]$ and $med\{1_j, u', v'\}=1_j=\bar{v}$ we have $z'=1_j$. In both cases we obtain $(z, z')\in E(G)$.

\end{enumerate}

\noindent When $u, w $ are the vertices of a square with exactly one reflexive vertex $0_j$ and irreflexive vertices $1_j, 2_j, 3_j$ we have the following extra case.

\begin{enumerate}\setcounter{enumi}{2}

\item 
 $\bar{w}=w=u'$,  $\bar{u}=u=w'$ and $\bar{v}=v'$, with  $(u, w)\in E(G)$. This case happens when $u=3_j$, $w=0_j$ and $v$ is such that $d(k_v^{\alpha}, p_{X_i})\le d(k_{3_j}^{\alpha}, p_{X_i})$ in $G_T$ (note that in this case we must have $k_{3_j}^{\alpha}=3_j$, and $k_v^{\alpha}\in \{1_j,2_j\}$). 

Let $x=med\{0_j, 3_j, v\}= k_v^{\alpha}$.
If $[3_j]=[v]=[x]$ then we must have $k_v^{\alpha}=1_j$ and so  $z=m(0_j, 3_j, v)=1_j$.
If  $[3_j]=[v]\neq [x]$ then we must have $k_v^{\alpha}=2_j$ and  $z=m(0_j, 3_j, v)= x^+=1_j$. In both cases we have $[3_j]\neq [\bar{v}]$ which implies that $z'=p_{X_i}=0_j$, thus $(z, z')=(1_j, 0_j)\in E(G)$.
If $[3_j]\neq [v]$ then $z= 0_j$, and we have $[\bar{u}]=[3_j]=[\bar{v}]$. If $[med\{0_j, 3_j, \bar{v}\}]=[3_j]$ (which happens when $med\{0_j, 3_j, \bar{v}\}=1_j$) we get $z'= med\{0_j, 3_j, \bar{v}\}=1_j$. If  $[med\{0_j, 3_j, \bar{v}\}]\neq [3_j]$ 
(this happens when $med\{0_j, 3_j, \bar{v}\}=2_j$) we get $z'=(med\{0_j, 3_j, \bar{v}\})^+= 2_j^+=1_j$. Thus, in both cases $(z, z')=(0_j, 1_j)\in E(G)$.

\end{enumerate}

\noindent When $u, w $ are the vertices of a square with exactly two reflexive vertices $0_j, 3_j$ and irreflexive vertices $1_j, 2_j $ we have the following extra cases.

\begin{enumerate}\setcounter{enumi}{3}  
\item $\bar{w}=u'=w, \bar{u}=w'=u$ and $\bar{v}=v'$. This happens when $w=3_j, u=2_j$ and $k_v^t=1_j$ with $P_t$ the path beginning with the edges $(0_j,1_j),(1_j,2_j),(2_j,d) \in E(G)$ (with $d \in X_i$). 
We then have 
    $$z=m(2_j, 3_j, v)=\left\{ \begin{array}{cl}
    2_j, &{\rm if}\ [v]=[2_j];\\
    3_j, &{\rm if}\ [v]\neq [2_j],\end{array}\right.
    $$
    and $$z'=m(3_j, 2_j, v')=\left\{ \begin{array}{cl}
    3_j, &{\rm if}\ [v']\neq [2_j]\Leftrightarrow [v]=[2_j];\\
    2_j, &{\rm if}\ [v']=[2_j] \Leftrightarrow [v]\neq [2_j],\end{array}\right.$$
    as $d(w, 3_j)=d(\bar{w}, 3_j)< d(\bar{w}, 0_j)=d(w, 0_j)$.

\item $\bar{w}=u', \bar{v}=w'$ and $\bar{u}=v'$. This can happen when $u=2_j, w=0_j$ and $k_v^t=1_j$. Then, as $d(w, 0_j)< d(w, 3_j)$, we  have always have 
$z=m(2_j, 0_j, v)=0_j$ independently of $[v]$ being equal to $[2_j]$ or not, since $(1_j)^+=0_j$. Also
$$z'=m(3_j, 1_j, v')=\left\{ \begin{array}{cl}
    1_j, &{\rm if}\ [v']=[1_j];\\
    3_j, &{\rm if}\ [v']\neq [1_j], \end{array}\right.$$
implying that we always have $(z, z')\in E(G)$.
    
\item $\bar{u}=u', \bar{v}=w'$ and $\bar{w}=v'$. This may occur when $w=3_j$, $k^t_v=1_j$ and $k^t_u=2_j$.
We then have $med\{u,v,w\}=1_j$  $$z=m(u, v, 3_j)=\left\{ \begin{array}{cl}
    1_j, &{\rm if}\ [u]=[v]=[1_j];\\
    0_j, &{\rm if}\ [u]=[v]\neq[1_j];\\
    3_j, &{\rm if}\ [u]\neq [v],\end{array}\right.
    $$
     as $d(w,3_j)<d(w,0_j)$. 

     We can see that $med\{u',v',2_j\}=2_j$. Hence $$z'=m(u',v', 2_j)=\left\{ \begin{array}{cl}
    2_j, &{\rm if}\ [\bar{u}]=[\bar{v}]=[2_j] \Leftrightarrow [u]=[v]=[1_j] ;\\
    1_j, &{\rm if}\ [\bar{u}]=[\bar{v}]\neq[2_j] \Leftrightarrow [u]=[v]\neq[1_j] ;\\
    0_j, &{\rm if}\ [\bar{u}]\neq [\bar{v}]\Leftrightarrow [u]\neq[v],  \end{array}\right.$$ as $d(\bar{w},0_j)<d(\bar{w},3_j)$, implying that we always get $ (z,z')\in E(G)$.
     
\end{enumerate}

\subsubsection{Case 1.}\label{case1}

Assume that  $a, b, c\in P_{t}\subseteq X_i$ for some  $i=1, \ldots, k$  and some $t=1, \ldots, n_i$   we  define $u, v, w$ such that $d(u, p_{X_i})\geq d(v, p_{X_i})\geq d(w, p_{X_i})$;

In this case we have 
$$z=m(u,v,w)=\begin{cases}
    v, & \text{if } [u]=[v];\\
    w, & \text{if } [u]\neq[v],
    \end{cases}$$

since $w\in X_i$ and $x=med\{a,b,c\}=v$ (and $w=k_w^j$).

If $med\{a', b', c'\}\in A$ or exactly two of $a', b', c'$ belong to $X_i$ (or to another irreflexive component $X_l$) then we are in Cases 0. and 4., respectively, and the proof is symmetric to what we did in Subsections~\ref{case0} and \ref{case4}. So we can assume that $a',b', c'\in X_i$. 

 Let $u',v',w' \in \{a',b',c'\}$ be such that  $(u, u'),(v,v'),(w,w')\in E(G)$, and let  $\bar{u},\bar{v},\bar{w} \in \{u',v',w'\}$ be the vertices such that the ordering with respect to the appropriate indexed path, $P_{\alpha}$, is $d(k_{\bar{u}}^{\alpha}, p_{X_i})\geq d( k_{\bar{v}}^{\alpha}, p_{X_i}) \geq d(k_{\bar{w}}^{\alpha}, p_{X_i})$. For each of the cases  $P_{\alpha}$ will be the fixed lowest index path chosen accordingly.

\noindent\boxed{\text{Case}\ 1.1.} Suppose first that $z=m(a, b, c)=m(u, v, w)=v$. So we have $[u]=[v]$. Looking at the order of $u', v', w'$ compared with $u, v, w$ we have the following possibilities:
\begin{enumerate}
\item  $\bar{u}=u', \bar{v}=v'$ (or $\bar{v}=u', \bar{u}=v'$) and $\bar{w}=w'$; in this case we have $[\bar{u}]=[\bar{v}]$ and $med\{\bar{u}, \bar{v}, \bar{w}\}\in \{\bar{v}, k_{\bar{v}}^{\alpha}\}$.  
When $\bar{v} \in P_{\alpha}$ we have  $med\{\bar{u}, \bar{v}, \bar{w}\}=\bar{v}$, and so $z'=\bar{v}$. When $\bar{v} \notin P_{\alpha}$ we have  $med\{\bar{u}, \bar{v}, \bar{w}\}= k_{\bar{v}}^{\alpha}$. Because $u, v, w$ are all on the same path, if the median of $\{\bar{u}, \bar{v}, \bar{w}\}$ is not $\bar{v}$ it must be $v=k_{\bar{v}}^{\alpha}=k_{\bar{v}}^t$. In this case $[\bar{v}]\neq [med\{\bar{u}, \bar{v}, \bar{w}\}]$ so we get $z'=m(\bar{u}, \bar{v}, \bar{w})= (k_{\bar{v}}^{\alpha})^+=v^+$, yielding $(z, z')\in E(T)$.

\item $\bar{u}=u', \bar{v}=w'=v$ and $\bar{w}=v'$; here  $(v, w)\in E(G)$ and 
 $[u']\neq [v]$, so $[\bar{u}]\neq[\bar{v}]$. Also since $u, v, w\in P_t$  we must have  that $\bar{u}, \bar{v} \in P_{\alpha}$, which implies that   $k^{\alpha}_{\bar{w}} =\bar{w}$.  It follows that $z'=m(a',b',c')=(k^{\alpha}_{\bar{w}})=\bar{w}=v'$, and we have $(z, z')\in E(T)$.
\end{enumerate}

\noindent\boxed{\text{Case}\  1.2.} Suppose now that   $z=m(a, b, c)=m(u, v, w)=w$. So we have $[u]\neq [v]$, and we look at the order of $u', v', w'$ in relation to $u, v, w$:
\begin{enumerate}
\item  $\bar{u}=u', \bar{v}=v'$ (or $\bar{v}=u', \bar{u}=v'$) and $\bar{w}=w'$; in this case we have $[\bar{u}]\neq[\bar{v}]$. As we are assuming that $u', v', w'\in X_i$  we must have that either $\bar{w}\in P_{\alpha}$ or $w=k_{\bar{w}}^{\alpha}$.  If $\bar{w}=w'\in P_{\alpha}$ we get  $[\bar{w}]=[k_{\bar{w}}^{\alpha}]$ implying that  $z'=\bar{w}=k_{\bar{w}}^{\alpha}$.
If $\bar{w}\notin P_{\alpha}$ then we get $k_{\bar{w}}^{\alpha}=w$, and $[\bar{w}]\neq [k_{\bar{w}}^{\alpha}]$ implying that  $z'=(k_{\bar{w}}^{\alpha})^+= w^+$.
In both cases we get that $(z, z')\in E(T)$.

\item  $\bar{u}=u', \bar{v}=w'=v$ and $\bar{w}=v'$; since $[u]\neq [v]$ we have $[u']= [v]$, so $[\bar{u}]=[\bar{v}]$. From $v=med \{u, v, w\}$ and $\bar{v}=v$ we get that $med \{\bar{u}, \bar{v}, \bar{w}\}= \bar{v}$, thus $z'=\bar{v}$. Since this only happens if $(v, w)\in E(T)$ we get $(z, z')\in E(T)$.
\end{enumerate}

We have two more cases that can occur  when $u, w$ are vertices of  an irreflexive square $C_l$.  We look at these for Cases 1.1. and 1.2. together.

\noindent The first is when
 $\bar{u}=u', \bar{v}=w'=3_l$ and $\bar{w}=v'$.
 
Here  $w\in \{0_l, 2_l\}$ and we obtain 
$$z'=m(\bar{w}, 3_l, \bar{u})= m(v', 3_l, u') =\left\{\begin{array}{cl}
    \bar{w}, & {\rm if} \ [\bar{u}]\neq [3_l]; \\
    3_l, & {\rm if} \ [\bar{u}]= [3_l].\end{array}\right.$$

    If $v=1_l$ then $[u]=[v]\Leftrightarrow [u]=[1_l](=[3_l])\Leftrightarrow [u']\neq [3_l]\Leftrightarrow [\bar{u}]\neq [3_l]$, and
    $[u]\neq [v]\Leftrightarrow [u]=[2_l]\Leftrightarrow [\bar{u}]=[3_l]$. So in this case we have $z=v \Leftrightarrow z'=\bar{w}=v'$ or $z=0_l\Leftrightarrow z'= 3_l$. 

    If $v=2_l$ then we,  similarly,  get $[u]=[v]\Leftrightarrow [u]=[2_l]\Leftrightarrow [\bar{u}]=[3_l]$ and $[u]\neq [v]\Leftrightarrow [u]=[3_l]\Leftrightarrow [\bar{u}]\neq [3_l]$, which implies that we get $z= v=2_l \Leftrightarrow z'= 3_l$ or $z=0_l\Leftrightarrow z' =\bar{w}=v'\in \{1_l, 3_l\}$.  

    If $v=3_l$ then we must have that $u=v$. Hence $z=3_j$, and $[u]=[v]\Leftrightarrow [\bar{u}]\neq[\bar{v}]$, giving $z'=\bar{w}=v'$.
    
    Thus, in all cases we obtain $(z, z')\in E(G)$.

The second is when $\bar{u}=v'$, $\bar{w}=u'=w=0_l$, and $\bar{v}=w'=u=3_l$.

When $[u]=[v]$ (Case 1.1.), recall that $med \{u, v, w\}=v$.
Since $[u]=[v]$ we must have  $z=v=1_l$. It then follows that  $[\bar{u}]=[v']\neq [\bar{v}]=[u]$  and $z'=m(\bar{u}, \bar{v}, \bar{w})=w=0_l$ because $\bar{w}=k_{\bar{w}}^{\alpha}$. Thus $(z, z')=(1_l, 0_l)\in E(G)$.

When $[u]\neq[v]$ (Case 1.2.) we must have  $z=w=0_l$, which implies that $v=2_l$. It then follows that $[\bar{v}]=[\bar{u}]$, and since $med\{\bar{u}, \bar{v}, \bar{w}\}=med \{0_l, 3_l, v'\}=v=2_l$ we have that $z'=m(\bar{u}, \bar{v}, \bar{w})=v^+=(2_l)^+=1_l$. Thus $(z, z')=(0_l, 1_l)\in E(G)$.

\subsubsection{Case 2.}

Let $a, b, c \in X_i$ but such that  only two belong to the same path $P_t$, for some $i=1, \ldots, k$ and $t=1, \ldots n_i$. Choose $t$ to be the smallest possible index satisfying this, and define $u, v, w$ such that  $d(k_u^t, p_{X_i})\geq d(k_v^t, p_{X_i})\geq d(k_w^t, p_{X_i})$.
Recall that $u=k_{u}^t$, so either $v\notin P_{t}$ and $v\neq k_v^t$,  or $w\notin P_{t}$ and $w\neq k_w^t$.

With $x=med\{a, b, c\}$, we have the following possibilities:

 \begin{equation*}
   z= m(a,b,c)=m(u,v,w)=\begin{cases}
x ,& \text{if \ }  [u]=[v]=[x]; \\
x^+,& \text{if \ } [u]=[v]\neq[x];  \\
k^t_w, & \text{if \ } [u]\neq [v] \text{\ and \ } [w]=[k^t_w]; \\
(k^t_w)^+, & \text{if \ } [u]\neq [v] \text{\ and \ } [w]\neq[k^t_w]. \\
\end{cases}
\end{equation*}

 Let  $u',v',w'\in \{a', b', c'\}$ be such that $(u, u'),(v, v'),(w,w')\in E(T)$. If $med\{u', v', w'\}\notin X_i$ or only two of $u', v', w'$ are in $X_i$ then we are in Case 0. or Case 4. 
 and the proof follows symmetrically to what was done in Subsections~\ref{case0} and \ref{case4} respectively. So we can assume that $u', v', w'\in X_i$ and that they do not all belong to the same path, otherwise we would be in Case 1. and the proof would follow like in Subsection~\ref{case1}.  Let  $\bar{u},\bar{v},\bar{w} \in \{u',v',w'\}$  be such that the ordering with respect to the appropriate indexed path, $P_{\alpha}$, is $d(k_{\bar{u}}^{\alpha}, p_{X_i})\geq d( k_{\bar{v}}^{\alpha}, p_{X_i}) \geq d(k_{\bar{w}}^{\alpha}, p_{X_i})$. 

\bigskip

\noindent\boxed{\rm{Case}\ 2.1.} Suppose first that  $z=m(a, b, c)=x=med\{a,b,c\}$. So we know that $[u]=[v]=[x]$. 

If $v\notin P_t $ then $x=k_v^t= med\{u, v, w\}$. As $[u]=[v]$ these vertices must be at distance at least two from each other and because $[v]=[k_v^t]$ the vertex $v$ must be at distance at least two from the path $P_t$. 
So either we have  $\bar{u}=u'$, $\bar{v}=v'$, and $\bar{w}=w'$; or $\bar{w}=u'=0_l, \bar{v}=v', \bar{u}=w'=3_l$, with $\bar{w}, \bar{u}$ vertices of an irreflexive square $C_l$.
In the former case  we have $[\bar{u}]=[\bar{v}]$ and $med\{\bar{u}, \bar{v}, \bar{w}\}=x$, implying that $[\bar{u}]=[\bar{v}]\neq [x]$. Thus $z'= m(\bar{u}, \bar{v}, \bar{w})= x^+$, and  $(z, z')\in E(G)$. Note that this holds even when we go into the case with non distinct $k^{\beta}$$  's$ for all paths $P_\beta$ as in Figure~\ref{fig:casethreept2}.
And in the latter case we have $P_t$ containing the vertices $0_l, 1_l, 2_l, 3_l$, and to ensure that $[u]=[v]=[k_v^j]$ we must have that $k_v^j=1_l$.   It then follows, as above, that $k_{\bar{v}}^j=1_l$ and $med\{u, v, w\}= med \{\bar{u}, \bar{v}, \bar{w}\}=1_l$, implying that $z=1_l$. As $\bar{u}=u$ we then have $[\bar{u}]\neq [\bar{v}]$, yielding $z'= \bar{w}$, as $\bar{w}\in P_{\alpha}$.  Hence $z'=\bar{w}=0_l$ and $(z, z')\in E(G)$.

If $v\in P_t$, and so $w\notin P_t$, we have $x= med\{u, v, w\}=v$.  As $[u]=[v]$ these vertices must be at distance at least two from each other or be equal. If $\bar{v}=v'$ and $\bar{w}=w'$ then $med\{\bar{u}, \bar{v}, \bar{w}\}\in \{ \bar{v}, v\}$, and so we have $z'=m(\bar{u}, \bar{v}, \bar{w})\in \{\bar{v}, v^+\}$, 
thus $(z, z')\in E(G)$. If this is not the case, then we have one of the following: 

\begin{enumerate}
\item $\bar{w}=v', \bar{v}=w'$ and $\bar{u}=u'$. This can happen when $k_w^t=k_{\bar{v}}^{\alpha}=v'=\bar{w}$, which implies that $med\{\bar{u}, \bar{v}, \bar{w}\}=v'=\bar{w}$. Since $\bar{w}=k_{\bar{w}}^{\alpha}$ we will always get $z'=m(\bar{u},\bar{v}, \bar{w})=\bar{w}=v'$, thus $(z, z')\in E(G)$. 
\end{enumerate}
 When the path $P_t$ contains the vertices $0_l, 1_l, 2_l, 3_l$ of an irreflexive square $C_l$ we can also have: 

 \begin{enumerate}\setcounter{enumi}{1}
 
 \item $\bar{w}=u'=0_l$, $\bar{v}=w'$ and $\bar{u}=v'$, with $u=3_l$. 
In this case because $[u]=[v]=[x]$ and $k^t_w=0_l$ we must have $v=1_l$ if $v\in P_t$ or $k^t_v=1_l$ otherwise,  so $z=m(w, 3_l, 1_l)=1_l$. Since $k_{\bar{v}}^{\alpha}=k_w^t=0_l$ we obtain $med \{\bar{u}, \bar{v}, \bar{w}\}=\bar{w}=0_l$. It follows that  we always get $z'=m(\bar{u}, \bar{v}, \bar{w})=0_l$, thus $(z, z')=(1_l, 0_l)\in E(G)$.

\item $\bar{w}=u'=0_l$, $\bar{v}=v'$ and $\bar{u}=w'$. This can happen when $u=3_l, v=1_l$ and $u'=v'=0_l$, which implies that $z'=m(w', 0_l, 0_l )=0_l$. Since $z=m(w, 3_l, 1_l)=1_l$, as $med\{1_l, 3_l, w\}=1_l$, we get that $(z, z')=(1_l, 0_l)\in E(G)$.

\end{enumerate}

\bigskip

\noindent\boxed{\rm{Case}\ 2.2.} Suppose now  that $m(a,b,c)=m(u, v, w)=x^+$, where $x=med\{a,b,c\}$. So we know that $[u]=[v]\neq[x]$.
As $[u]=[v]$ these vertices must be at distance at least two from each other. Also, since $[v]\neq [x]$ we must have that $v\notin P_t$ as otherwise this would be the median point. It follows that $w\in P_t$, and $x=k_v^t$.

If $\bar{u}=u', \bar{v}=v'$ then we have $med\{u', v', w'\}= k_v^t= x$ and  $[\bar{u}]=[\bar{v}]=[x]$, yielding $z'=m(a', b', c')=x$. Thus $(x^+, z')=(x^+, x)\in E(G)$.

If this is not the case then we can have one of the following cases:
\begin{enumerate}

\item $\bar{v}=w',\bar{u}=u'$ and $\bar{w}=v'$; when $P_\alpha=P_t$ (here we abuse notation as we only require these paths to coincide from the vertex $u$ to $P_{X_i}$)
, this case happens when  $w=k_v^t$, so $x=w$. 
We can then have $med\{u', v', w'\}= w'=\bar{v}$. Because $[v]\neq[x]$ we have  $[v]\neq [w]$ and  $[v']\neq [w']$ so 
$[u']\neq [w']\Leftrightarrow [\bar{u}]\neq [\bar{v}]$.
From  $k_{\bar{w}}^{\alpha}= k_{\bar{w}}^t= k_v^t=w$ we obtain   $[k_{\bar{w}}^{\alpha}]=[w]\neq [w']$, implying that  $[k_{\bar{w}}^{\alpha}]=[v']=[\bar{w}]$. It then follows that $z'=m(a', b', c')=k_{\bar{w}}^{\alpha}=w$, hence $(x^+, z')=(w^+, w)\in E(T)$.

When $P_\alpha\neq P_t$, the three vertices would be on different paths with non distinct $k^{\beta}$$'s$ for any path $P_\beta$. Note that this case cannot occur as $[\bar{u}]\neq [\bar{v}]$.

\item  $\bar{u}=w'$ and $ \bar{w}=u'$; this happens when $k_v^t=w=u'=\bar{w}$, $u=w'=\bar{u}$ and $\bar{v}=v'$. So that $k_{\bar{w}}^{\alpha}=k_{\bar{w}}^t=\bar{w}$.
In this case we have $x=med\{u, v, w\}= med\{u', v', w'\}=w=u'=\bar{w}$, and $[\bar{u}]=[u]= [v]$, and as $\bar{v}=v'$, we have $[v]\neq [\bar{v}]$, so $[\bar{u}]\neq [\bar{v}]$. Yielding $z'=m(u', v', w')= k_{\bar{w}}^{\alpha}=\bar{w}=x$, hence $(x^+, z')=(x^+, x)\in E(T)$. Note that in this case, from $u$ to $P_{X_i}$, $P_\alpha= P_t$.

\end{enumerate}
Note that $\bar{u}=v'$ and $\bar{v}=u'$ only happens when $u'=v'=x$,  since $[u]=[v]$. This implies  that $z'=m(x, x, w')=x$, but we could relabel vertices  as $\bar{u}=u'$ and $\bar{v}=v'$, so this was considered above. 
When $P_j$ is the path containing the vertices $0_l, 1_l, 2_l, 3_l$ of an irreflexive square $C_l$ we can also have the following cases:

\begin{enumerate}\setcounter{enumi}{2}

\item $\bar{w}=w'$, $\bar{v}=u'=0_l$ and $\bar{u}=v'$; this  might happen when $u=3_l$ and $k_v^t\in \{0_l, 2_l\}$, ($k^t_v \neq 1_l $ as $[u]\neq[x]$) which implies that $z\in \{(0_l)^+, 1_l\}$. In both cases we have $[u']=[v']\Leftrightarrow [\bar{u}]=[\bar{v}]$, and $med\{\bar{u}, \bar{v}, \bar{w}\}=0_l=\bar{v}$. Thus $z'=0_l$, implying that $(z, z')\in E(G)$.

\item $\bar{u}=w'=3_l$, $\bar{v}=v'$, and $\bar{w}=u'=0_l$, which can happen when $u=3_l, w=0_l$ and $k_v^t\in \{0_l,2_l\}$, yielding $z\in \{(0_l)^+, 1_l\}$. In this case we have $[u]=[3_l]=[\bar{u}]=[v]$ , so $[\bar{u}]\neq [\bar{v}]$. Since $\bar{w}=0_l=k_{\bar{w}}^{\alpha}$ we get $z'=m(\bar{u}, \bar{v}, \bar{w})=0_l$, thus $(z, z')\in E(G)$.
\end{enumerate}

\bigskip 

\noindent\boxed{\rm{Case}\ 2.3.} Suppose now that $m(a,b,c)=m(u, v, w)=k_w^t$. So we know that $[u]\neq [v]$ and $[w]=[k_w^t]$.

\begin{enumerate}
\item If $\bar{u}=u', \bar{v}=v'$, and $\bar{w}=w'$ (or $\bar{v}=u', \bar{u}=v'$, and $\bar{w}=w'$) then  $[\bar{u}]\neq [\bar{v}]$.

When $w\notin P_t$, because $[w]=[k_w^t]$ we must also have $\bar{w}\notin P_t$, and so $k_{\bar{w}}^{\alpha}= k_{w}^t$, which implies that $[\bar{w}]\neq [k_{\bar{w}}^{\alpha}]$. We then obtain $z'=m(a', b', c')= (k_{\bar{w}}^{\alpha})^+= (k_w^t)^+$, which implies that $(k_w^t, 
z')=(k_w^t, (k_w^t)^+)\in E(T)$.

When $w\in P_t$ then $k_w^t=w$, and either $\bar{w}\in P_t$ which would imply that $k_{\bar{w}}^{\alpha}=k_{\bar{w}}^t=\bar{w}$, and so $z'=m(a', b', c')=\bar{w}$; 
or $\bar{w}\notin P_t$ which would imply that 
$k_{\bar{w}}^{\alpha}=k_{\bar{w}}^t=w$, and so  $z'=m(a', b', c')= w^+$. In both cases we have $(k_w^t, z')\in E(T)$.

\item If $\bar{u}=u', \bar{v}=w'$, and $\bar{w}=v'$ then we must have $v\notin P_t$,  
$w=k_v^t=k_w^t$, and $(w, v)\in E(T)$. It then follows that $[\bar{u}]= [\bar{v}]\neq [\bar{w}]$ and $med\{\bar{u}, \bar{v}, \bar{w}\}= \bar{w}$. Then $z'=m(a', b', c')= \bar{w}^+ = (k_w^t)^+$. Thus $(z, z')=(k_w^t, (k_w^t)^+)\in E(T)$.

\item If $\bar{u}=v'$, $\bar{v}=w'$ and $\bar{w}=u'$ then $w=k^t_v=k^t_w \in P_t$ and $w'\notin P_t$. We can see that $z=w=k^t_w$ and $med\{u',v',w'\}=w'$. When $[\bar{u}]=[\bar{v}]$, then $z'=w'$. When $[\bar{u}]\neq[\bar{v}] \Leftrightarrow [\bar{w}]\neq[k^t_w]$, and as $k^t_w=k^\alpha_{\bar{w}}$, we get $z'=(k^t_w)^+$. Hence $(z,z')\in E(G)$ in all cases. 
This is similar when $\bar{u}=w'$ and $\bar{v}=u'$.

\end{enumerate}

When two of $a, b, c$ are in the path containing the vertices $0_l, 1_l, 2_l, 3_l$ of an irreflexive square $C_l$ in $X_i$, we can also have:

\begin{enumerate}\setcounter{enumi}{3}
\item $\bar{u}=w'=3_l$, $\bar{w}=u'=0_l$ and $\bar{v}=v'$, which can happen when $u=3_l, w=0_l$. We then have $z=k_{0_l}^t=0_l$, and $[\bar{u}]=[\bar{v}]$, as $\bar{u}=u$ and $\bar{v}=v'$. We also have that $x'=med \{0_l, 3_l, \bar{v}\}=k_{\bar{v}}^{\alpha}\in \{0_l, 1_l, 2_l\}$. If $x'\in \{0_l,2_l\}$ then $[\bar{u}]\neq [x']$ and we get $z'\in \{(0_l)^+, (2_l)^+=1_l\}$. If $x'=1_l$ then $[\bar{u}]=[x']$ and we get $z'=x'=1_l$. Hence in all cases we get $(z, z')\in E(G)$.

\item $\bar{w}=u'=0_l$, $\bar{v}=w'$, and $\bar{u}=v'$, which can happen when $u=3_l$.
If $v\in P_j$  and $w\not\in P_j$ then we must have $v=2_l$ and $k_w^t\in\{0_l,1_l\}$. We have $[\bar{u}]=[v']\neq [2_l]$ and $[\bar{v}]=[w']\neq [k_w^t]$. If $k_w^t=0_l$ we have $z=0_l$, and  $[\bar{u}]=[\bar{v}]\neq[0_l]$ and $med\{\bar{u}, \bar{v}, \bar{w}\}=0_l$, yielding $z'=(0_l)^+$.  
If $k_w^t=1_l$ we have $z=1_l$, and $[\bar{v}]=[0_l]$, so $[\bar{u}]\neq [\bar{v}]$ and $[\bar{w}]=[0_l]=[k_w^t]$, yielding $z'=k_{\bar{w}}^{\alpha}=\bar{w}=0_l$. In both cases we get $(z, z')\in E(G)$.

\end{enumerate}

\noindent\boxed{\rm{Case} \ 2.4.} Suppose now that $m(a,b,c)=m(u, v, w)=(k_w^t)^+$. So we know that $[u]\neq [v]$ and $[w]\neq [k_w^t]$.

From $[w]\neq [k_w^t]$ we know that $w\notin P_t$, so we can assume that $v\in P_t$ and $P_{\alpha}=P_t$, from $u$ to $P_{X_i}$.
This implies that,  we cannot have $(v, w)\in E(G)$, hence, if $u$ is not a specific vertex on an irreflexive square,  we  have $\bar{w}= w'$.

 It then follows that $[\bar{u}]\neq [\bar{v}]$, independently of $\bar{u}=u', \bar{v}=v'$, or $\bar{u}=v', \bar{v}=u'$. 
 As $[w]\neq [\bar{w}]$ we get that  $[k_{\bar{w}}^t]=[\bar{w}]$, and so   $z'=m(a', b', c')= k_{\bar{w}}^t$. Hence $((k_w^t)^+, z')=((k_w^t)^+, k_w^t)\in E(T)$.

When $u=3_l$ is a vertex on an irreflexive square $C_l$, we can have  $u'=0_l$ and  $k_w^t=0_l$. We will then  get $\bar{w}=u'=0_l$, $\bar{v}=w'$, and $\bar{u}=v'$. Because $[u]\neq [v]$ this case only happens when $v=2_l$, which implies that $z=m(2_l, 3_l, w)=(0_l)^+$. Now 
$[\bar{u}]=[v']=[1_l]$ and $[\bar{v}]=[w']=[0_l]$, since $[w]\neq[ 0_l]$. Also, $med \{\bar{u}, \bar{v}, \bar{w}\}=0_l$, which yields $z'=0_l$, hence $(z, z')=((0_l)^+, 0_l)\in E(G)$.

\subsubsection{Case 3.} 
Suppose now that  $a, b, c$ are  all on different paths, say $P_j$,$P_k$,$P_m$, then we take $t=min \{j,k,m\}$, and  $u,v,w$ such that $d(k^t_u,p_{X_i}) \geq d(k^t_v,p_{X_i}) \geq d(k^t_w,p_{X_i})$. If two of $k^t_u$, $k^t_v$, $k^t_w$ are equal, then we take the next lowest path to base this on i.e. $P_s$ with $s=min\{\{j,k,m\} \,\backslash\, t\}$, and if this distinction is not possible we take $u, v$ to be  two of the vertices in the same colour class.

We have the following possibilities:

 \begin{equation*}
   z= m(a,b,c)=m(u,v,w)=\begin{cases}
x ,& \text{if \ }  [u]=[v]=[x]; \\
x^+,& \text{if \ } [u]=[v]\neq[x];  \\
k^t_w, & \text{if \ } [u]\neq [v] \text{\ and \ } [w]=[k^t_w]; \\
(k^t_w)^+, & \text{if \ } [u]\neq [v] \text{\ and \ } [w]\neq[k^t_w], \\
\end{cases}
\end{equation*}
with $x=med\{a, b, c\}$.
 Let  $u',v',w'\in \{a', b', c'\}$ be such that $(u, u'),(v, v'),(w,w')\in E(G)$.  Let  $\bar{u},\bar{v},\bar{w} \in \{u',v',w'\}$  be such that the ordering with respect to the appropriate indexed path, $P_{\alpha}$, is $d(k_{\bar{u}}^{\alpha}, p_{X_i})\geq d( k_{\bar{v}}^{\alpha}, p_{X_i}) \geq d(k_{\bar{w}}^{\alpha}, p_{X_i})$. 

As all the other cases have already been considered, and the proof would follow symmetrically, we can assume that   $u', v', w'\in X_i$ all belong to different paths. When we have a   choice of $\alpha$, this implies 
that we must have $\bar{u}=u'$,  $\bar{v}=v'$ (or $\bar{u}=v'$, $\bar{v}=u'$) and  $\bar{w}=w'$. If we do not have a distinguishable choice of $\alpha$, then the order can only change when all three vertices are in the same colour class, and so we can treat this as if it has not changed at all. 

Also, it is worth noting that the even though the vertices $u, v, w$ might be in this case and be some of the vertices of an irreflexive square, if $\bar{u}, \bar{v}, \bar{w}$ change order we will get at least two elements on the same path, so we can solve this case symmetrically to one of the ones above.

\noindent\boxed{\rm{Case}\  3.1.} Suppose first that $m(a,b,c)=m(u, v, w)=x$, so  that $[u]=[v]=[x]$. 
Because $u, v, w$ are all on different paths, and we know that $u\in P_t$, we can see that $x=med\{u,v,w\}=k_v^t$.

Then we have  $[u]=[v]\neq [med\{\bar{u}, \bar{v}, \bar{w}\}]$, since $med\{\bar{u},\bar{v}, \bar{w}\}= med\{u, v, w\}=x$. It then follows that $z'=m(a', b', c')=x^+$, and so $(z, z')\in E(T)$.

\bigskip

\noindent\boxed{\rm{Case}\  3.2.} Suppose now that $m(a,b,c)=m(u, v, w)=x^+$, so that $[u]=[v]\neq[x]$.

We have  $[\bar{u}]=[\bar{v}]= [med\{\bar{u}, \bar{v}, \bar{w}\}]$, since $med\{\bar{u},\bar{v}, \bar{w}\}= med\{u, v, w\}=x$. It then follows that $z'=m(a', b', c')=x$, and so $(z, z')\in E(T)$.

\bigskip

\noindent\boxed{\rm{Case}\ 3.3.} Suppose now that $m(a,b,c)=m(u, v, w)=k_w^t$. So we know that $[u]\neq [v]$ and $[w]=[k_w^t]$.
Then  $[\bar{u}]\neq [\bar{v}]$ and,  as $P_{\alpha}$ is identical to $P_t$ from $k^t_v$, we 
have  $k_{\bar{w}}^{\alpha}= k_{w}^t$, which implies that $[\bar{w}]\neq [k_{\bar{w}}^{\alpha}]$. We then obtain $z'=m(a', b', c')= (k_{\bar{w}}^{\alpha})^+= (k_w^t)^+$, which implies that $(k_w^t, z')=(k_w^t, (k_w^t)^+)\in E(G)$.

\bigskip

\noindent\boxed{\rm{Case} \ 3.4.} Finally, suppose that $m(a,b,c)=m(u, v, w)=(k_w^t)^+$. So we know that $[u]\neq [v]$ and $[w]\neq [k_w^t]$. We have $[\bar{u}]\neq [\bar{v}]$, and 
$k_{\bar{w}}^{\alpha}= k_{w}^t$, which implies that $[\bar{w}]= [k_{\bar{w}}^{\alpha}]$. We then obtain $z'=m(a', b', c')= (k_{\bar{w}}^{\alpha})^+= k_w^t$, which implies that $((k_w^t)^+, z')=( (k_w^t)^+, k_w^t)\in E(G)$.

This proves Proposition~\ref{proofmaj}.

\subsection{Reflexive and irreflexive cacti}\label{irreflexive}

We note that when our cactus graph $G$ is fully reflexive then  the only cycles that it can contain are (reflexive) triangles. In this case the majority operation that preserves the graph can be taken to be the median in the spanning tree $G_T$ obtained by removing the edge $(1_l,2_l)$ from each triangle $C_l$. This majority generalises the one given in \cite{feder2003bi} for reflexive trees to cacti.

When the cactus $G$ is fully irreflexive the only cycles it can contain are (irreflexive) squares. In this case we do not have a corresponding reflexive vertex, $p_{X_i}$, that we used to define $m$ in (\ref{majorityfunction}). But we can still define a very similar majority if we fix a vertex in $G$ that is a vertex with degree 2 in a square, or a leaf when no squares are present. Calling this vertex $p$, we can then define a majority function, just like (\ref{majorityfunction}), noting that 
the role of $p_{X_i}$ is taken by $p$, e.g. the vertex $a^+$ is the neighbour of $a$ in the spanning tree $G_T$ closest to $p$. We must also define a specific neighbour of $p$ to be $p^+$ for when $k^t_w=p$ and $[w]\neq[k^t_w]$. Note that the spanning tree is defined just like before, by removing from $G$ the edges $(0_i, 3_i)$, where $0_i$ is the vertex of the square $C_i$ closest to $p$ (or the vertex $p$ itself). As before the $x=med\{a, b, c\}$ taken in $G_T$,  and we have the following restriction of (\ref{majorityfunction}):

\begin{equation*}
    m(a,b,c)=\begin{cases}  
x, & \text{if }\  [u]=[v]=[x]; \\
x^+,& \text{if \ } [u]=[v]\neq[x];  \\
k^t_w, & \text{if \ } [u]\neq [v] \text{\ and\ } [w]=[k^t_w]; \\
(k^t_w)^+, & \text{if \ } [u]\neq [v] \text{\ and\ } [w]\neq[k^t_w], \\
\end{cases}
\label{majorityfunctionirreflexive}
\end{equation*}

We are then in Cases 1., 2. or 3. and the proof that this function is a polymorphism  of  $G$ follows like the proof of Proposition~\ref{proofmaj}. 

\subsection{Extending the majority to forests of cacti.}

When $G$ is a forest of cacti, with connected cactus components $H_1, \ldots, H_z$, we can extend the symmetric majority  operation $m$ defined in (\ref{majorityfunction}) to $G$ as follows: for each connected component $H_i$ we fix a reflexive vertex, or, when $H_i$ has no reflexive vertices,  a vertex with degree 2 in a square, or a leaf when no squares are present (just like in previous subsection). Call $q_i$ this fixed vertex in $H_i$. We then define  and operation $m':G^3\longrightarrow G$ to be  symmetric, and to satisfy for any $x, y, z\in V(G)$:
\begin{enumerate} 
\item if $x, y, z\in H_i$ for some $i=1, \ldots, z$ then  $m'(x, y, z)=m(x, y, z)$ as defined in (\ref{majorityfunction}); 

\item if $x, y, z$  belong to three distinct connected components $H_{i}, H_j, H_k$, then $m'(x, y, z)=a$ with $a\in \{x, y, z\}\cap  H_s$ and $s=\min\{i, j, k\}$; 

\item if $x\in H_i$ and $y, z\in H_j$, with $i\neq j$ then $$m'(x, y, z)=\left\{
\begin{array}{cl}

x, & {\rm if \ } [y]\neq [z] \ {\rm and } \ H_j \ {\rm has\ no \ reflexive \ vertices;} \\
m(q_j, y, z), & {\rm otherwise.}\\

\end{array}. \right.$$

\end{enumerate}
We note that the colour class $[y]$ is defined just like at the start of this section.

\begin{claim}
The function $m'$ is a majority polymorphism of the forest  of cacti $G$. 
\end{claim}

\begin{proof}
The fact that $m'$ is a majority function follows from the fact that $m$ defined in (\ref{majorityfunction}) is a majority function.

To show that $m'$ is edge-preserving, the only  case that does not follow from previous subsections is when we have vertices $x, y, z\in V(G)$ such that $x\in H_i, y, z\in H_j$ and $i\neq j$. So we assume this is the case for $x, y, z$.

Let $x',  y', z'$ be such that $(x, x'), (y, y'), (z, z')\in E(G)$. So we have  $x'\in H_i$ and $y', z'\in H_j$. 

If $[y]\neq [z]$ then we have $[y']\neq [z']$, which yields that  $m'(x, y, z)=x$ and $m'(x', y', z')=x'$.

If $[y]=[z]$ then we have $[y']=[z']$, yielding that $m'(x, y, z)=m(q_j, y, z)$ and $m'(x', y', z')=m(q_j, y', z')$.
We want to show that $(m(q_j,y,z),m(q_j,y',z'))\in E(G)$.

When $q_j$ is a reflexive vertex this follows from Proposition~\ref{proofmaj}.

Suppose now that $H_j$ has no reflexive vertices, so $q_j$ is a vertex of degree two in a square or a leaf in $H_j$. By relabelling the vertices $x, y, z$ with $u, v, w$ like in Section~\ref{irreflexive}, consider first when $y,z$ are $u$ and $v$ (in either order) - note that the other option is that one or both of  $y, z$ equal $q_j$, as with the way the order is defined we will always have $w=w'=q_j$. We know that $[y]=[z]$,  so $m(q_j,y,z)\in \{med\{q_j,y,z\}, (med\{q_j,y,z\})^+\}$.

First assume that $y',z'$ are equal to $\bar{u}$ and $\bar{v}$ (in any way) in the appropriate ordering, as defined in Section~\ref{irreflexive}. If $q_j,y,z$ are on the same path then $[y]=[z]=[med\{q_j,y,z\}]$ and $m(q_j,y,z)=med\{q_j,y,z\}$. If $y',z'$ stay on the same path then $[y']=[z']=[med\{q_j,y',z'\}]$, and $m(q_j,y',z')=med\{q_j,y',z'\}$, and the claim follows. Otherwise, one or both of $y',z'$ move off of the path and we have  $med\{q_j,y,z\}=med\{q_j,y',z'\}$ and $[y']=[z']\neq[med\{q_j,y',z'\}]$. We then get   $m(q_j,y',z')=(med\{q_j,y',z'\})^+$, which is adjacent to $m\{q_j,y,z\}$.
When $q_j,y,z$ are not on the same path, the case when $y',z'$ go onto a path is symmetric to above, so we only need to consider when $q_j,y',z'$ are also not on a path. Here, $med\{q_j,y,z\}=med\{q_j,y',z'\}$, and so $m(q_j,y,z)\in \{med\{q_j,y,z\}, (med\{q_j,y,z\})^+\}$ dependent on the colour class of the median and $m(q_j,y',z')\in \{(med\{q_j,y',z'\})^+, med\{q_j,y',z'\}\}$, respectively. 
Suppose now that  one (or both) of $y',z'$ are equal to $q_j$ and $[y]\neq[q_j]$. Hence, we have $m(q_j, y, z)=med\{q_j,y,z\}$, which is equal to the vertex of $y,z$ that is adjacent to $q_j$ and $m(q_j,y',z')=med\{q_j,y',z'\}=q_j$, which proves the claim. 

The only thing left to  consider is when $q_j$ is one of $u$ or $v$. Note that this can only happen when at least one of $y,z$ equals  $q_j$. Without loss of generality assume that $y=q_j$. Hence, here $m(q_j,y,z)=q_j=med\{q_j,y,z\}$, and $[q_j]=[y]=[z]$.  As $[y']=[z']$ we have that $y', z'\neq q_j$. 
We then have $y'$ and $z'$ equal to $\bar{u'}$ and $\bar{v'}$ (in either order). 

From the way this case is built, we must have that $med\{q_j,y',z'\}\in \{q_j,y'\}$ so in all possibilities for $m(q_j, y', z')$ we will always have that $(m(q_j,y,z),m(q_j,y',z'))\in E(G)$.
\end{proof}

\section{Conclusions}

The main question arising from our work is how far the classification can be pushed. To some extent, we already answered this question. A \emph{$k$-cactus} is a connected graph where every edge belongs to at most $k$ cycles. Thus, a $1$-cactus is a cactus, and $K_4$, in which exactly one vertex is looped, is a $2$-cactus. As we explained, this contains irreflexive $K_3$, yet has a majority polymorphism. Thus, our classification does not extend even to $2$-cacti.

It is clear that the search for which reflexive cacti have a non-crossing semilattice polymorphism as in \cite{hell2014semilattice} is trivial, and these are exactly cacti in which every cycle is a triangle. A natural question is then to ask which irreflexive and partially reflexive cacti have such.

Another question arising is if it is true that every (undirected) graph with  a majority polymorphism also has a symmetric majority polymorphism.  This is not true for digraphs, as  for example the (partially reflexive) digraph  in Figure 6 has a majority polymorphism but not a symmetric majority polymorphism.

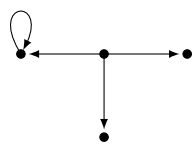
\begin{figure}[H]
    \centering
    \begin{tikzpicture}[transform shape]
	\node[circ](0) at (3.9, 10){};
	\node[circ] at (5, 10){};
	\node[circ] at (6.1, 10){};
	\node[circ] at (5, 8.9){};
	\draw[-latex] (5, 10) -- (6, 10);
	\draw[-latex] (5, 10) -- (5, 9);
	\draw[-latex] (5, 10) -- (4, 10);
    \draw [-latex] (0)  to [min distance=8mm,out=120,in=60,in looseness=7] (0);

\end{tikzpicture}
    \caption{A digraph with majority but no symmetric majority.}
    \label{fig:placeholder}
\end{figure}
Also, this is not true for $k$-NUs, with $k>3$, as for example the irreflexive undirected square (cycle on four vertices) has a $4$-NU polymorphism but not a symmetric $4$-NU polymorphism; It does though have a symmetric majority. The graph represented in Figure 7 is  an example of an irreflexive graph that is preserved by a $4$-NU, but not by a  symmetric $4$-NU, and is also not preserved by a majority polymorphism. These polymorphisms were tested using the Graph Polymorphisms Applet \cite{maroti}. 
The question remains if every undirected graph with a majority polymorphism also has a symmetric majority polymorphism.

\begin{figure}[H]
    \centering
    \resizebox{0.15\textwidth}{!}{
    \begin{tikzpicture}[transform shape]
	\node[circ] at (5, 11){};
	\node[circ] at (3, 10){};
	\node[circ] at (5, 9){};
	\node[circ] at (7, 10){};
	\node[circ] at (5, 7){};
	\node[circ] at (3, 8){};
	\node[circ] at (7, 8){};
	\draw (3, 10) to[short] (3, 8);
	\draw (3, 8) to[short] (5, 7);
	\draw (5, 7) to[short] (7, 8);
	\draw (7, 8) to[short] (7, 10);
	\draw (5, 11) to[short] (7, 10);
	\draw (5, 11) to[short] (3, 10);
	\draw (3, 10) to[short] (5, 9);
	\draw (5, 9) to[short] (7, 10);
	\draw (5, 9) to[short] (5, 7);
\end{tikzpicture}}
    \caption{A graph with $4$-NU, no majority, and no symmetric $4$-NU.}
    \label{fig:placeholder}
\end{figure}
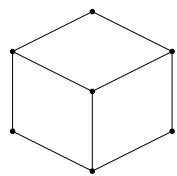

\bibliography{cspgraph}
\bibliographystyle{plain}
\end{document}